\documentclass[]{article}

\usepackage{amsmath,amsfonts,amssymb}
\usepackage{amsthm}
\usepackage{mathrsfs}
\usepackage{indentfirst}
\usepackage{booktabs}
\usepackage{multirow}
\usepackage{tabularx}
\usepackage{subfigure}
\usepackage{cite}
\usepackage{graphicx}
\usepackage{cases}
\usepackage{float}
\usepackage[colorlinks,linkcolor=blue,anchorcolor=black,citecolor=blue]{hyperref}
\usepackage[hmargin=1.25in,vmargin=1in]{geometry}
\theoremstyle{definition}
\newtheorem{lemma}{Lemma}[section]
\newtheorem{theorem}[lemma]{Theorem}

\newtheorem{proposition}[lemma]{Proposition}
\numberwithin{equation}{section}
\title{Mixed virtual element methods for elliptic optimal control problems with boundary observations in $L^{2}(\Gamma)$}
\author{Minghui Yang, Zhaojie Zhou}
\author{Minghui Yang$^\ddagger$, Zhaojie Zhou$^\ddagger$}
\date{$^\ddagger$School of Mathematics and Statistics, Shandong Normal University, Ji’nan 250014, China}
\begin{document}
	\maketitle
	\begin{abstract}
In this paper we study the mixed virtual element approximation to an elliptic optimal control problem with boundary observations. The objective functional of this type of optimal control problem contains the outward normal derivatives of the state variable on the boundary, which reduces the regularity of solutions to the optimal control problems. We construct the mixed virtual element discrete scheme and derive a priori error estimate for the optimal control problem based on the variational discretization for the control variable. Numerical experiments are carried out on different meshes to support our theoretical findings.

	\end{abstract}
\section{Introduction}
	In this paper we consider the following elliptic optimal control problems with boundary observations:
	\begin{equation}\label{1.1}
		\underset{u\in U_{ad}}{\mbox{min}} J(y,u)=\frac{1}{2}\int_{\Gamma} (\partial_{n_{A}}y-y_{d})^{2}ds+\frac{\gamma}{2}\int_{\Omega}u^{2}dx
	\end{equation}
	subject to
	\begin{equation}\label{1.2}
		\left\{
		\begin{aligned}
			&-\mbox{div}(A\nabla y)=f+u &&\mbox{in}\ \ \Omega,\\
			&y=g &&\mbox{on}\ \ \Gamma,
		\end{aligned}
		\right.
	\end{equation}
	where $y_{d}\in L^{2}(\Gamma)$ and $f\in L^{2}(\Omega)$ are given functions, $\gamma>0$ is the regularization parameter, and $\Omega$ is a convex and polygonal domain in $\mathbb{R}^{2}$ with the boundary $\Gamma=\partial\Omega$, $\partial_{n_{A}}y$ is defined as $A\nabla y\cdot\boldsymbol{n}$ in which $\boldsymbol{n}$ is the outward normal vector on the boundary. The admissible control set $U_{ad}$ is defined by
	\begin{equation}
		U_{ad} =\{ u \in L^{2}(\Omega):\ {a}\le u\le {b}\mbox{ a.e. in }\Omega\},
	\end{equation}
where ${a}$ and ${b}$ are real numbers. Suppose that the matrix $A$ is symmetric, positive definite and, for simplicity, constant, such that there is a constant $c>0$ satisfying
\begin{equation}
	\boldsymbol{v}^{T}A\boldsymbol{v}\ge c\Vert \boldsymbol{v}\Vert_{\mathbb{R}^{2}}^{2} \quad \forall \boldsymbol{v}\in \mathbb{R}^{2}.
\end{equation}

PDE-constrained optimal control problems play increasingly important role in physics, aerospace and engineering. In many optimal control problems, the objective functional usually contains distributed observations of the state. However, in many practical applications, boundary observation is more widely used (see \cite{1,2}). This problem can be viewed as a PDE-constrained optimal control problem with distributed control and observations on the boundary. In \cite{3} the authors considered an elliptic optimal control problems with pointwise observation of the state and the boundary observations were introduced in \cite{4}. In \cite{5} the authors studied the Neumann boundary control of elliptic equations with Dirichlet boundary observations. In this paper we intend to consider the elliptic optimal control problems with inhomogeneous dirichlet boundary condition with boundary observations in $L^{2}(\Gamma)$. The objective functional contains the outward normal derivatives of the state variable on the boundary, which reduces the regularity of solutions to the optimal control problems.

As systems of optimal control problems often contain partial differential equations, it is difficult to obtain their exact solution. Therefore, it is important to solve the PDE-constraints numerically for optimal control problems. In the past decade, numerical method for PDE-constrained optimal control problems is a hot research topic. Many different numerical methods are used to solve optimal control problems, such as the finite element methods \cite{6,7}, the mixed finite element methods \cite{8,9}, the spectral methods \cite{10}, the local discontinuous Galerkin method \cite{11} and so on.

Virtual element method (VEM) was introduced in \cite{vem}, which was born as an extension of the finite element to polygons. Compared to traditional finite element methods, the VEM avoids the explicit construction of the local basis functions and has the capability of dealing with general meshes. Moreover, the VEM has the possibility to easily implement highly regular discrete spaces \cite{13,14}. The method uses approximate bilinear form to look for the solution in the discretization space. Generally, the discrete bilinear form is divided into the sum of two components: the first component is usually constructed using a projection operator to ensure consistency, and the numerical accuracy is guaranteed; the second part is a bilinear form usually added to recover stability, vanishing for polynomials and ensuring the well posedness of the discrete problem. In recent years, the VEM has also been widely studied. For example, the nonconforming VEM \cite{15,16}, the mixed VEM \cite{mvem,mvem2}, the VEM for curved edge \cite{19,trace}, and so on. Recently, the VEM has also been applied to solve optimal control problems. A $C^{1}$ virtual element method for an elliptic distributed optimal control problem with pointwise state constraints was proposed in \cite{21}. In \cite{22,23} the authors considered a posteriori error analysis for virtual element discretization of second order elliptic optimal control problems. A mixed VEM for second order optimal control problems was studied in \cite{24}. As far as we know, apart from distributed observations, there are only few papers on error analysis of optimal control problems with boundary observations.

In this paper, we aim to derive a priori error estimate for mixed virtual element approximation to the optimal control problem \eqref{1.1}-\eqref{1.2}. For the boundary observation problems, the major challenge comes from the fact that the adjoint state equation is an elliptic equation with inhomogeneous dirichlet boundary condition whose value is the partial derivative of the state. To ease this diffculty, we adopt the mixed methods which make the essential boundary condition appears to be a natural one and it is easier for theoretical analysis and numerical implementation. Based on a mixed variational scheme for second order elliptic equations \cite{mvem2}, we convert the state equation to the mixed system. Since the optiaml control problem with boundary observations have limited regularity, the state and adjoint state variables are approximated by the lowest order mixed virtual element space. The variational discretization proposed in \cite{vd} is used to the control variable. We derive the error estimates for the mixed virtual element approximation under weaker assumptions on the regularity of the solution which reach the optimal order estimates. Finally, we provide some numerical examples on different meshes to confirm our theoretical findings.

This paper is organized as follows. In Section 2 we introduce the optimal control problem in details. In Section 3 we derive its first order optimality condition and we present the mixed virtual element analysis of the problem in Section 4. Finally we give some numerical examples in Section 5 to support our results.

\section{Preliminaries}
In this section, we introduce some preliminary knowledge, including the mixed weak form and the continuous first order optimality condition for optimal control problem. In this paper, we adopt the standard notations $W^{m,p}(\Omega)$ for Sobolev spaces on $\Omega$ with norm $\Vert\cdot\Vert_{m,p,\Omega}$ and semi-norm $\vert\cdot\vert_{m,p,\Omega}$ and denote $W^{m,2}$ by $H^{m}(\Omega)$ with norm $\Vert\cdot\Vert_{m,\Omega}$. Let $(\cdot,\cdot)$ and $\langle\cdot,\cdot\rangle$ represent the $L^{2}$-inner products on $L^{2}(\Omega)$ and $L^{2}(\Gamma)$ respectively. In addition, $C$ denote generic positive constants. Then define the space $H(\mbox{div},\Omega)$ as
\begin{equation*}
	H(\mbox{div},\Omega):=\{\boldsymbol{v}\in (L^{2}(\Omega))^{2},\ \mbox{div}\ \boldsymbol{v}\in L^{2}(\Omega)\}
\end{equation*}
with norms
\begin{equation*}
	\Vert\boldsymbol{v}\Vert_{H(\text{div},\Omega)}^{2}:=\Vert\boldsymbol{v}\Vert_{0,\Omega}^{2}+\Vert\text{div }\boldsymbol{v}\Vert_{0,\Omega}^{2}.
\end{equation*}
Through the rest of this paper, we take $\boldsymbol{V}=H(\mbox{div},\Omega)$ and $W=L^{2}(\Omega)$ to carry out anlysis. By introducing a flux variable $\boldsymbol{p}=A\nabla y$, we can convert \eqref{1.2} into the following mixed form
\begin{equation}\label{2.1}
	\left\{
	\begin{aligned}
		\boldsymbol{p}&=A\nabla y,\\
		-\text{div }\boldsymbol{p}&=f+u.
	\end{aligned}
    \right.
\end{equation}
Then we can write \eqref{2.1} in variational form: find $(\boldsymbol{p},y)\in \boldsymbol{V}\times W$ such that,
\begin{equation}\label{2.2}
	\left\{
	\begin{aligned}
		a(\boldsymbol{p},\boldsymbol{v})+b(y,\boldsymbol{v})&=\langle g,\boldsymbol{v}\cdot\boldsymbol{n}\rangle &&\forall\boldsymbol{v}\in\boldsymbol{V},\\
		b(w,\boldsymbol{p})&=-(f+u,w) &&\forall w\in W,
	\end{aligned}
	\right.
\end{equation}
where the blinear form 
\begin{equation*}
	\begin{aligned}
		a(\boldsymbol{\sigma},\boldsymbol{v})&:=\int_{\Omega}A^{-1}\boldsymbol{\sigma}\boldsymbol{v}dx  &&\forall\boldsymbol{\sigma},\boldsymbol{v}\in\boldsymbol{V},\\
		b(\omega,\boldsymbol{\sigma})&:=\int_{\Omega}\omega\cdot\mbox{div }\boldsymbol{\sigma}dx\quad &&\forall\boldsymbol{\sigma}\in\boldsymbol{V},\omega\in W,
	\end{aligned}
\end{equation*}
and
\begin{equation*}
	\langle g,\boldsymbol{v}\cdot\boldsymbol{n}\rangle:=\int_{\Gamma}g\cdot\boldsymbol{v}\cdot\boldsymbol{n}ds\quad \forall\boldsymbol{v}\in\boldsymbol{V}.
\end{equation*}
Therefore the mixed weak form of optimal control problem \eqref{1.1}-\eqref{1.2} is as follows
\begin{equation}\label{2.3}
	\underset{u\in U_{ad}}{\mbox{min}} J(\boldsymbol{p},u)=\frac{1}{2}\int_{\Gamma} (\boldsymbol{p}\cdot\boldsymbol{n}-y_{d})^{2}ds+\frac{\gamma}{2}\int_{\Omega}u^{2}dx
\end{equation}
subject to
\begin{equation}\label{2.4}
	\left\{
	\begin{aligned}
		a(\boldsymbol{p},\boldsymbol{v})+b(y,\boldsymbol{v})&=\langle g,\boldsymbol{v}\cdot\boldsymbol{n}\rangle &&\forall\boldsymbol{v}\in\boldsymbol{V},\\
		b(w,\boldsymbol{p})&=-(f+u,w) &&\forall w\in W.
	\end{aligned}
	\right.
\end{equation}

\begin{lemma}\label{lem2.1}
\textit{Define a trace operator} $\mathscr{R}:H(\text{div},\Omega)\to H^{-\frac{1}{2}}(\Gamma)$ \textit{such that $\mathscr{R}(\boldsymbol{v}):=\boldsymbol{v}\cdot\boldsymbol{n}$ in the sense that}
\begin{equation*}
	\int_{\Gamma}\boldsymbol{v}\cdot\boldsymbol{n}\omega ds:=\int_{\Omega}(\mbox{div }\boldsymbol{v})\omega dx+\int_{\Omega}\boldsymbol{v}\cdot\nabla\omega dx\qquad\forall \omega\in H^{1}(\Omega).
\end{equation*}
\textit{The operator $\mathscr{R}$ is a surjection from} $H(\text{div},\Omega)$ \textit{to $H^{-\frac{1}{2}}(\Gamma)$ and continuous such that}
\begin{equation*}
	\Vert\boldsymbol{v}\cdot\boldsymbol{n}\Vert_{-\frac{1}{2},\Gamma}\le\Vert\boldsymbol{v}\Vert_{H(\text{div},\Omega)}.
\end{equation*}
\end{lemma}
\begin{proof}
	For the proof we can refer to \cite{26}.
\end{proof}
From \cite{4,ocp}, the optiaml control problems \eqref{1.1}-\eqref{1.2} are equivalent to the following boundary value problems in the sense of distribution:
\begin{equation}\label{2.5}
	\begin{aligned}
		&-\mbox{div}(A\nabla y)=f+u &&\mbox{in}\ \ \Omega, &&&&y=g &&&&&\mbox{on}\ \ \Gamma,\\
		&-\mbox{div}(A\nabla z)=0 &&\mbox{in}\ \ \Omega, &&&&z=-(\partial_{n_{A}}y-y_{d}) &&&&&\mbox{on}\ \ \Gamma,\\
		&\int_{\Omega}(\gamma u+z)(v-u)dx\le 0, &&\forall v\in U_{ad}.
	\end{aligned}
\end{equation}

By the standard arguments, we can derive the weak form of first-order optimality condition as follows: find $(\boldsymbol{p},y,\boldsymbol{r},z,u)\in\boldsymbol{V}\times W\times\boldsymbol{V}\times W\times U_{ad}$ such that
\begin{equation}\label{2.6}
	\begin{aligned}
		a(\boldsymbol{p},\boldsymbol{v})+b(y,\boldsymbol{v})&=\langle g,\boldsymbol{v}\cdot\boldsymbol{n}\rangle  &&\forall \boldsymbol{v}\in\boldsymbol{V},\\
		b(w,\boldsymbol{p})&=-(f+u,w) &&\forall w\in W,\\
		a(\boldsymbol{r},\boldsymbol{v})+b(z,\boldsymbol{v})&=\langle y_{d}-\boldsymbol{p}\cdot\boldsymbol{n}, \boldsymbol{v}\cdot\boldsymbol{n}\rangle  &&\forall \boldsymbol{v}\in\boldsymbol{V},\\
		b(w,\boldsymbol{r})&=0 &&\forall w\in W,\\
		\int_{\Omega}(\gamma u-z)(v-u)dx&\ge 0  &&\forall v\in U_{ad}.\\
	\end{aligned}
\end{equation}
By intrduce a projection operator, the variational inequality of \eqref{2.6} can be written as
\begin{equation}\label{2.7}
	u=P_{U_{ad}}\bigg\{-\frac{1}{\gamma}z\bigg\},
\end{equation}
where $P_{U_{ad}}(\cdot):=\text{max}\{a, \text{min}\{\cdot, b\}\}$ is the $L^{2}$-orthogonal projection onto $U_{ad}$.

At the end of the section, we consider the regularity of solutions to the optimal control problem \eqref{1.1}-\eqref{1.2}.
\begin{lemma}\label{lem2.2}
	\textit{Let $\Omega$ be an open bounded convex polygonal domain with a Lipschitz continuous boundary $\Gamma$. Assume that $f\in L^{2}(\Omega),g\in H^{\frac{3}{2}}(\Gamma)$ and $y_{d}\in H^{\frac{1}{2}}(\Gamma)$. Let $u\in U_{ad}$ be the solution of \eqref{2.6} with corresponding state $y$ and adjoint state $z$. Then we have}
	\begin{equation*}
		y\in H^{2}(\Omega), \quad z\in H^{1}(\Omega), \quad u\in H^{1}(\Omega).
	\end{equation*}
\end{lemma}
\begin{proof}
	From $f+u\in L^{2}(\Omega)$, $g\in L^{\frac{3}{2}}(\Gamma)$ and the assumptions on the domain $\Omega$ and $A$, we can conclude from the elliptic regularity theory \cite{FEM} that $y\in H^{2}(\Omega)$ and then $\partial_{n_{A}}y\in H^{\frac{1}{2}}(\Gamma)$. Therefore we can derive $z\in H^{1}(\Omega)$ under the assumption on $y_{d}$ and the regularity of the adjoint equation \eqref{2.5} with $z\vert_{\Gamma}=-(\partial_{n_{A}}y-y_{d})\in H^{\frac{1}{2}}(\Gamma)$. From \eqref{2.7}, we can obtain $u\in H^{1}(\Omega)$ directly.
\end{proof}

\section{Mixed virtual element approximation}
In this section, we aim to construct the discrete scheme of optimal control problem. Let $\Omega$ be a convex polygonal domain and $\mathcal{T}_{h}$ be a partitioning of $\Omega$ into disjoint polygon $E$. Moreover, let $h_{E}$ denote the diameter of $E$, $h=\underset{E\in\mathcal{T}_{h}}{\mbox{max}}\ h_{E}$. Then, we introduce the assumption of mesh regularity in virtual element method.

\vspace{8pt}
\noindent\textbf{Assumption 3.1. } The family of decompositions $\mathcal{T}_{h}$ satisfies    e that $E$ is star-shaped, for every element $E$, there exists a $\rho_{E}>0$ such that\\
$\bullet$ $E$ is star-shaped with respect to a ball of radius $\rho_{E}h_{E}$,\\
$\bullet$ the ratio between the length $h_{e}$ of every edge $e$ of $E$ and the diameter $h_{E}$ of $E$ is bigger than $\rho_{E}$.
\vspace{8pt}

\noindent Follow \cite{mvem2}, we introduce some useful definaition. For every element $E\in\mathcal{T}_{h}$, define
\begin{equation*}
	\boldsymbol{\mathcal{M}}_{k}(E):=\nabla\mathbb{P}_{k+1}(E),
\end{equation*}
and
\begin{equation*}
	\boldsymbol{\mathcal{M}}_{k}^{\bot}(E):=\text{the } L^{2}(E)\ \text{orthogonal of }\boldsymbol{\mathcal{M}}_{k}(E) \text{ in }(\mathbb{P}_{k}(E))^{2}.
\end{equation*}
Further, we have
\begin{equation*}
	(\mathbb{P}_{k}(E))^{2}= \boldsymbol{\mathcal{M}}_{k}(E)\oplus\boldsymbol{\mathcal{M}}_{k}^{\bot}(E).
\end{equation*}
For integer $k\ge 0$, according to the defination of \cite{mvem2}, the mixed virtual element space can be defined as
\begin{equation*}
	\begin{aligned}
		&\boldsymbol{V}_{h}=\{\boldsymbol{v}\in\boldsymbol{V}:\ \boldsymbol{v}\vert_{E}\in\boldsymbol{V}_{h}^{E},\ \forall E\in \mathcal{T}_{h}\},\\
		&W_{h}=\{w\in W:\ w\vert_{E}\in\mathbb{P}_{k}(E),\ \forall E\in\mathcal{T}_{h}\},
	\end{aligned}
\end{equation*}
where
\begin{equation*}
	\begin{aligned}
		&\boldsymbol{V}_{h}^{E}:=\{\boldsymbol{v}_{h}\in H(\mbox{div}, E)\cap H(\mbox{rot}, E)\ \mbox{such that}\\
		&\qquad\qquad(\boldsymbol{v}_{h}\cdot\boldsymbol{n})\vert_{e}\in \mathbb{P}_{k}(e)\ \mbox{for all edge\ }e\in\partial E,\ (\mbox{div }\boldsymbol{v}_{h})\vert_{E}\in\mathbb{P}_{k}(E)\ \mbox{and}\ (\mbox{rot }\boldsymbol{v}_{h})\vert_{E}\in\mathbb{P}_{k-1}(E)\}.
	\end{aligned}
\end{equation*}

The degrees of freedom for $W_{h}$ are obvious, while the degrees of freedom for $\boldsymbol{V}_{h}$ are defined by (see \cite{mvem2})
\begin{equation}\label{3.1}
	\begin{aligned}
		&\int_{e}\boldsymbol{v}\cdot\boldsymbol{n}q_{k}ds  &&\text{for all edge }e,\ \text{for all }q_{k}\in\mathbb{P}_{k}(e),\\
		&\int_{E}\boldsymbol{v}\cdot\boldsymbol{m}_{k-1}dx  &&\text{for all element }E,\ \text{for all }\boldsymbol{m}_{k-1}\in\boldsymbol{\mathcal{M}}_{k-1}(E),\\
		&\int_{E}\boldsymbol{v}\cdot\boldsymbol{m}^{\bot}_{k}dx  &&\text{for all element }E,\ \text{for all }\boldsymbol{m}^{\bot}_{k}\in\boldsymbol{\mathcal{M}}^{\bot}_{k}(E).
	\end{aligned}
\end{equation}
According to \cite{mvem2}, the interpolation operator $\boldsymbol{\Pi}_{h}$ from $(H^{1}(\Omega))^{2}\to\boldsymbol{V}_{h}$ can be defined through the
degrees of freedom \eqref{3.1}
\begin{equation}\label{3.2}
	\left\{
	\begin{aligned}
		&\int_{e}(\boldsymbol{v}-\boldsymbol{\Pi}_{h}\boldsymbol{v})\cdot\boldsymbol{n}q_{k}ds=0  &&\text{for all edge }e,\ \text{for all }q_{k}\in\mathbb{P}_{k}(e),\\
		&\int_{E}(\boldsymbol{v}-\boldsymbol{\Pi}_{h}\boldsymbol{v})\cdot\boldsymbol{m}_{k-1}dx=0  &&\text{for all element }E,\ \text{for all }\boldsymbol{m}_{k-1}\in\boldsymbol{\mathcal{M}}_{k-1}(E),\\
		&\int_{E}(\boldsymbol{v}-\boldsymbol{\Pi}_{h}\boldsymbol{v})\cdot\boldsymbol{m}^{\bot}_{k}dx=0  &&\text{for all element }E,\ \text{for all }\boldsymbol{m}^{\bot}_{k}\in\boldsymbol{\mathcal{M}}^{\bot}_{k}(E).
	\end{aligned}
\right.
\end{equation}
Similarly, we intrduce the interpolation of $W_{h}$
\begin{equation}\label{3.31}
	\int_{E}(\varphi-\Pi_{h}\varphi)q_{k}dx=0\quad \forall q_{k}\in\mathbb{P}_{k}(E).
\end{equation}

Since the optiaml control problem \eqref{1.1}-\eqref{1.2} have low regularity, the lowest order space associated with the partitioning $\mathcal{T}_{h}$ of $\Omega$ will be used. In order to obtain the discrete scheme of \eqref{2.4}, we introduce projection operator $\boldsymbol{\Pi}_{k}^{0}:H(\text{div},E)\to(\mathbb{P}_{k}(E))^{2}$ defined by
\begin{equation}\label{3.3}
	\int_{E}(\boldsymbol{v}-\boldsymbol{\Pi}_{k}^{0}\boldsymbol{v})\boldsymbol{p}_{k}dx=0\quad\forall \boldsymbol{p}_{k}\in(\mathbb{P}_{k}(E))^{2}.
\end{equation}
In \cite{vemprojection} it was shown that the degrees of freedom \eqref{3.2} allow the explicit computation of the projection $\boldsymbol{\Pi}_{k}^{0}\boldsymbol{v}$ from the knowledge of the degrees of freedom \eqref{3.2} of $\boldsymbol{v}$. From \cite{mvem2}, we have the following projection error estimate
\begin{equation}\label{3.4}
	\Vert\boldsymbol{v}-\boldsymbol{\Pi}_{k}^{0}\boldsymbol{v}\Vert_{0,\Omega}\le Ch^{s}\vert\boldsymbol{v}\vert_{s,\Omega}\quad 0\le s\le k+1.
\end{equation}

 From the definition of virtual element function space, we know $\text{div }\boldsymbol{v}_{h}\in\mathbb{P}_{k},\forall\boldsymbol{v}_{h}\in\boldsymbol{V}_{h}$ for every element $E\in\mathcal{T}_{h}$. Therefore, the bilinear form $b(\omega_{h},\boldsymbol{v}_{h})$ can be explicitly computed. On the other hand, from \cite{mvem2}, by using projection operator $\boldsymbol{\Pi}_{0}^{0}$ the discrete bilinear  form $a_{h}$ can be defined as
\begin{equation}\label{3.5}
	a_{h}^{E}(\boldsymbol{v}_{h},\boldsymbol{w}_{h}):=a^{E}(\boldsymbol{\Pi}_{0}^{0}\boldsymbol{v}_{h},\boldsymbol{\Pi}_{0}^{0}\boldsymbol{w}_{h})+S^{E}(\boldsymbol{v}_{h}-\boldsymbol{\Pi}_{0}^{0}\boldsymbol{v}_{h},\boldsymbol{w}_{h}-\boldsymbol{\Pi}_{0}^{0}\boldsymbol{w}_{h}),
\end{equation}
where $S^{E}(\boldsymbol{v}_{h},\boldsymbol{w}_{h})$ is any symmetric and positive definite bilinear form that scales like $a^{E}(\boldsymbol{v}_{h},\boldsymbol{w}_{h})$ (see \cite{vem}). In addition, we assume that $S^{E}$ statify the following conditions: there exist two positive constants $\alpha^{*}$ and $\alpha_{*}$ such that
\begin{equation}\label{3.6}
	\alpha_{*}a^{E}(\boldsymbol{v}_{h},\boldsymbol{w}_{h})\le S^{E}(\boldsymbol{v}_{h},\boldsymbol{w}_{h})\le \alpha^{*}a^{E}(\boldsymbol{v}_{h},\boldsymbol{w}_{h}).
\end{equation}
From \cite{mvem2}, the most natural VEM stabilization $S^{E}(\cdot, \cdot)$ is given by
\begin{equation}\label{3.7}
	S^{E}(\boldsymbol{v}-\boldsymbol{\Pi}_{0}^{0}\boldsymbol{v},\boldsymbol{w}-\boldsymbol{\Pi}_{0}^{0}\boldsymbol{w}):=\vert E\vert\sum_{i}\mbox{dof}_{i}(\boldsymbol{v}-\boldsymbol{\Pi}_{0}^{0}\boldsymbol{v})\mbox{dof}_{i}(\boldsymbol{w}-\boldsymbol{\Pi}_{0}^{0}\boldsymbol{w}),
\end{equation}
where dof$_{i}$ denote the $i$-th local degree of freedom.
\begin{proposition}
	(See \cite{mvem2}) \textit{The blinear form $a_{h}$ has the following two properties:\\
		$\bullet$\ consistency
		\begin{equation*}
			a^{E}_{h}(\boldsymbol{q}_{k},\boldsymbol{v}_{h})=a^{E}(\boldsymbol{q}_{k},\boldsymbol{v}_{h})\qquad \forall E\in\mathcal{T}_{h},\ \forall\boldsymbol{q}_{k}\in(\mathbb{P}_{k}(E))^{2},\ \forall\boldsymbol{v}_{h}\in\boldsymbol{V}_{h}^{E},
		\end{equation*}
		$\bullet$\ stability}\\
	\begin{equation*}
		\exists\mu_{*},\ \mu^{*}>0\quad \text{s.t.}\quad \mu_{*}a^{E}(\boldsymbol{v}_{h},\boldsymbol{v}_{h})\le a^{E}_{h}(\boldsymbol{v}_{h},\boldsymbol{v}_{h})\le \mu^{*}a^{E}(\boldsymbol{v}_{h},\boldsymbol{v}_{h})\ \ \forall E\in\mathcal{T}_{h},\ \forall\boldsymbol{v}_{h}\in\boldsymbol{V}_{h}^{E}.
	\end{equation*}
\end{proposition}
In a natural way, we define
\begin{equation*}
	a_{h}(\boldsymbol{v}_{h},\boldsymbol{w}_{h})=\sum_{E\in\mathcal{T}_{h}}a_{h}^{E}(\boldsymbol{v}_{h},\boldsymbol{w}_{h}).
\end{equation*}
\begin{lemma}\label{lem3.2}(See \cite{mvem2})
	\textit{The blinear form $a_{h}$ is continuous and elliptic in $(L^{2}(\Omega))^2$. There exists two positive numbers, for every $\boldsymbol{v}_{h},\boldsymbol{w}_{h}\in\boldsymbol{V}_{h}$ such that}
	\begin{equation*}
		\begin{aligned}
			\vert a_{h}(\boldsymbol{v}_{h},\boldsymbol{w}_{h})\vert&\le M\Vert\boldsymbol{v}_{h}\Vert_{0,\Omega}\Vert\boldsymbol{w}_{h}\Vert_{0,\Omega},\\
			\vert a_{h}(\boldsymbol{v}_{h},\boldsymbol{v}_{h})\vert&\ge \alpha\Vert\boldsymbol{v}_{h}\Vert_{0,\Omega}^{2}.\\
		\end{aligned}
	\end{equation*}
\end{lemma}
\begin{lemma}\label{lem3.3}(See \cite{mvem2})
	\textit{The discrete inf-sup condition holds}
	\begin{equation*}
		\underset{\boldsymbol{v}_{h}\in\boldsymbol{V}_{h}}{\text{sup}}\frac{b(\boldsymbol{v}_{h},w_{h})}{\Vert\boldsymbol{v}_{h}\Vert_{0,\Omega}}\ge\beta\Vert w_{h}\Vert_{0,\Omega}.
	\end{equation*}
\end{lemma}
Then we propose the mixed virtual element approximation of \eqref{2.3}-\eqref{2.4}.
\begin{equation}\label{3.8}
	\underset{u_{h}\in U_{ad}}{\mbox{min}} J(\boldsymbol{p}_{h},u_{h})=\frac{1}{2}\int_{\Gamma} (\boldsymbol{p}_{h}\cdot\boldsymbol{n}-y_{d})^{2}ds+\frac{\gamma}{2}\int_{\Omega}u_{h}^{2}dx
\end{equation}
subject to
\begin{equation}\label{3.9}
	\left\{
	\begin{aligned}
		a_{h}(\boldsymbol{p}_{h},\boldsymbol{v}_{h})+b(y_{h},\boldsymbol{v}_{h})&=\langle g,\boldsymbol{v}_{h}\cdot\boldsymbol{n}\rangle &&\forall\boldsymbol{v}_{h}\in\boldsymbol{V}_{h},\\
		b(w_{h},\boldsymbol{p}_{h})&=-(f+u_{h},w_{h}) &&\forall w_{h}\in W_{h}.
	\end{aligned}
	\right.
\end{equation}
For the discretization of the control variable we use the variational discretization concept presented in \cite{vd}. By the standard Lagrange multiplier theory the solution can be characterized by the following first order discrete optimality conditions: find $(\boldsymbol{p}_{h},y_{h},\boldsymbol{r}_{h},z_{h},u_{h})\in\boldsymbol{V}_{h}\times W_{h}\times\boldsymbol{V}_{h}\times W_{h}\times U_{ad}$ such that
\begin{equation}\label{3.10}
	\begin{aligned}
		a_{h}(\boldsymbol{p}_{h},\boldsymbol{v}_{h})+b(y_{h},\boldsymbol{v}_{h})&=\langle g, \boldsymbol{v}_{h}\cdot\boldsymbol{n}\rangle &&\forall\boldsymbol{v}_{h}\in\boldsymbol{V}_{h},\\
		b(w_{h},\boldsymbol{p}_{h})&=-(f+u_{h},w_{h}) &&\forall w_{h}\in W_{h},\\
		a_{h}(\boldsymbol{r}_{h},\boldsymbol{v}_{h})+b(z_{h},\boldsymbol{v}_{h})&=\langle y_{d}-\boldsymbol{p}_{h}\cdot\boldsymbol{n}, \boldsymbol{v}_{h}\cdot\boldsymbol{n}\rangle  &&\forall\boldsymbol{v}_{h}\in\boldsymbol{V}_{h},\\
		b(w_{h},\boldsymbol{r}_{h})&=0 &&\forall w_{h}\in W_{h},\\
		\int_{\Omega}(\gamma u_{h}-z_{h})(v-u_{h})dx&\ge 0  &&\forall v\in U_{ad}.\\
	\end{aligned}
\end{equation}

\section{A priori error estimates}
In this section we develop a priori error estimates for the mixed virtual element approximation. We let $\mathcal{S}_{h}$ denote the set of all edges in $\mathcal{T}_{h}$. This set can be divided into the set of boundary edges $\mathcal{S}_{h}(\partial\Omega):=\{e\in\mathcal{S}_{h},\ s\subseteq\partial\Omega\}$ and internal edges $\mathcal{S}_{h}(\Omega):=\{e\in\mathcal{S}_{h},\ s\subseteq\Omega\}$. For given $E\in\mathcal{T}_{h}$, we denote by $\mathcal{S}_{h}(E)\subset\mathcal{S}_{h}$ the set of edges of $E$. Given an edge $e\in\mathcal{S}_{h}$, we let $h_{e}$ be its length and let $\boldsymbol{n}_{e} :=(n_{1},n_{2})^{T}$ be the corresponding unit tangential vector along $e$. Firstly, we intrduce some useful conclusions which are important for deriving error estimates.
\begin{lemma}\label{lem4.1}(See \cite{mvem2})
	\textit{Let $(\varphi_{I},\boldsymbol{\psi}_{I})\in W_{h}\times\boldsymbol{V}_{h}$ be interpolant of $(\varphi,\boldsymbol{\psi})$ defined by \eqref{3.2} and \eqref{3.31}, respetively. Then we have the following approximation}
	\begin{equation*}
		\begin{aligned}
			\Vert \varphi-\varphi_{I}\Vert_{0,\Omega}&\le Ch\Vert \varphi\Vert_{1,\Omega}, &&\forall \varphi\in H^{1}(\Omega),\\
			\Vert \boldsymbol{\psi}-\boldsymbol{\psi}_{I}\Vert_{0,\Omega}&\le Ch\Vert \boldsymbol{\psi}\Vert_{1,\Omega}, &&\forall \boldsymbol{\psi}\in (H^{1}(\Omega))^{2},\\
			\Vert \text{div }(\boldsymbol{\psi}-\boldsymbol{\psi}_{I})\Vert_{0,\Omega}&\le Ch\Vert \text{div }\boldsymbol{\psi}\vert_{1,\Omega}, &&\forall \text{div }\boldsymbol{\psi}\in H^{1}(\Omega).
		\end{aligned}
	\end{equation*}
\end{lemma}
\noindent Further, according to \cite{mvep}, we have the following result.
\begin{lemma}\label{lem4.2}(See \cite{mvep})
	\textit{Let $\boldsymbol{\sigma}_{I}$ be interpolant of $\boldsymbol{\sigma}\in (H^{1}(\Omega))^{2}$ defined by \eqref{3.2}. Then the following result holds
	\begin{equation*}
		\Vert(\boldsymbol{\sigma}-\boldsymbol{\sigma}_{I})\cdot\boldsymbol{n}_{e}\Vert_{0,e}\le Ch^{\frac{1}{2}}\Vert\boldsymbol{\sigma}\Vert_{1,E}\quad\forall e\in \mathcal{S}_{h},
	\end{equation*}
where $E$ is any element of $\mathcal{T}_{h}$.}
\end{lemma}
Now we introduce the auxiliary problem: find $(\boldsymbol{p}_{h}(u),y_{h}(u),\boldsymbol{r}_{h}(u),z_{h}(u))\in\boldsymbol{V}_{h}\times W_{h}\times\boldsymbol{V}_{h}\times W_{h}$ such that
\begin{equation}\label{4.1}
	\begin{aligned}
		a_{h}(\boldsymbol{p}_{h}(u),\boldsymbol{v}_{h})+b(y_{h}(u),\boldsymbol{v}_{h})&=\langle g, \boldsymbol{v}_{h}\cdot\boldsymbol{n}\rangle &&\forall\boldsymbol{v}_{h}\in\boldsymbol{V}_{h},\qquad\qquad\\
		b(w_{h},\boldsymbol{p}_{h}(u))&=-(f+u,w_{h}) &&\forall w_{h}\in W_{h},\\
		a_{h}(\boldsymbol{r}_{h}(u),\boldsymbol{v}_{h})+b(z_{h}(u),\boldsymbol{v}_{h})&=\langle y_{d}-\boldsymbol{p}_{h}(u)\cdot\boldsymbol{n}, \boldsymbol{v}_{h}\cdot\boldsymbol{n}\rangle &&\forall\boldsymbol{v}_{h}\in\boldsymbol{V}_{h},\\
		b(w_{h},\boldsymbol{r}_{h}(u))&=0 &&\forall w_{h}\in W_{h}.
	\end{aligned}
\end{equation}
Moreover, let $(\boldsymbol{r}_{h}(y),z_{h}(y))\in \boldsymbol{V}_{h}\times W_{h}$ be the solution of the following auxiliary problem
\begin{equation}\label{4.2}
	\begin{aligned}
		a_{h}(\boldsymbol{r}_{h}(y),\boldsymbol{v}_{h})+b(z_{h}(y),\boldsymbol{v}_{h})&=\langle y_{d}-\boldsymbol{p}\cdot\boldsymbol{n}, \boldsymbol{v}_{h}\cdot\boldsymbol{n}\rangle &&\forall\boldsymbol{v}_{h}\in\boldsymbol{V}_{h},\\
		b(w_{h},\boldsymbol{r}_{h}(y))&=0 &&\forall w_{h}\in W_{h}.
	\end{aligned}
\end{equation}
Note that $(\boldsymbol{p}_{h}(u),y_{h}(u))$ is the standard mixed virtual element approximation of $(\boldsymbol{p},y)$. Using the error estimate introduced in \cite{mvem2}, we have the following results.
\begin{lemma}\label{lem4.3}
	\textit{Let $(\boldsymbol{p},y,\boldsymbol{r},z)$ be the solution of \eqref{2.6}, and let $(\boldsymbol{p}_{h}(u),y_{h}(u))$ be the solution of the auxiliary problem \eqref{4.1}. Assume that the conditions of Lemma} \ref{lem2.2} \textit{are valid. Then, we have}
	\begin{equation*}
		\begin{aligned}
			\Vert y-y_{h}(u)\Vert_{0,\Omega}&\le Ch\Vert y\Vert_{2,\Omega},\\
			\Vert \boldsymbol{p}-\boldsymbol{p}_{h}(u)\Vert_{H(\text{div},\Omega)}&\le Ch(\Vert y\Vert_{2,\Omega}+\Vert f+u\Vert_{1,\Omega}).
		\end{aligned}
	\end{equation*}
\end{lemma}
According to Lemma \ref{lem2.2} that we have only $z\in H^{1}(\Omega)$. Therefore, it is necessary to obtain the error estimate for the mixed virtual element approximation
under weaker assumptions on the regularity of the solution.
\begin{lemma}\label{lem4.4}
	\textit{Let $(\boldsymbol{p},y,\boldsymbol{r},z)$ be the solution of \eqref{2.6}, and let $(\boldsymbol{r}_{h}(u),z_{h}(u))$ be the solution of the auxiliary problem \eqref{4.2}. Assume that the conditions of Lemma} \ref{lem2.2} \textit{are valid. Then, we have}
	\begin{equation*}
		\Vert z-z_{h}(y)\Vert_{0,\Omega}\le Ch.
	\end{equation*}
\end{lemma}
\begin{proof}
Consider the following problem
\begin{equation}\label{4.3}
	\left\{
	\begin{aligned}
		-\text{div}(A\nabla\varphi)&=\xi &&\text{in }\Omega,\\
		\varphi&=0 &&\text{on }\partial\Omega,
	\end{aligned}
	\right.
\end{equation}
where $\xi\in L^{2}(\Omega)$. By the elliptic regularity $\varphi\in H^{1}_{0}(\Omega)\cap H^{2}(\Omega)$ (see \cite{pde}) and
\begin{equation*}
	\Vert\varphi\Vert_{2,\Omega}\le C\Vert \xi\Vert_{0,\Omega}.
\end{equation*}
Then we consider the mixed form of \eqref{4.3}. Find $(\boldsymbol{\psi},\varphi)\in H(\text{div};\ \Omega)\times L^{2}(\Omega)$ such that
\begin{equation*}
	\left\{
	\begin{aligned}
		a(\boldsymbol{\psi},\boldsymbol{v})+b(\varphi,\boldsymbol{v})&=0 &&\forall\boldsymbol{v}\in\boldsymbol{V},\\
		b(w,\boldsymbol{\psi})&=-(\xi,w) &&\forall w\in W,
	\end{aligned}
    \right.
\end{equation*}
and there holds
\begin{equation*}
	\Vert\varphi\Vert_{2,\Omega}+\Vert\boldsymbol{\psi}\Vert_{1,\Omega}\le C\Vert \xi\Vert_{0,\Omega}.
\end{equation*}
Then we have
\begin{equation*}
	\begin{aligned}
		(z-z_{h}(y),\xi)&=-b(z-z_{h}(y),\boldsymbol{\psi})-a(\boldsymbol{\psi},\boldsymbol{r}-\boldsymbol{r}_{h}(y))-b(\varphi,\boldsymbol{r}-\boldsymbol{r}_{h}(y))\\
		&=-b(z-z_{h}(y),\boldsymbol{\psi}-\boldsymbol{\psi}_{I})-a(\boldsymbol{\psi}-\boldsymbol{\psi}_{I},\boldsymbol{r}-\boldsymbol{r}_{h}(y))-b(\varphi-\varphi_{I},\boldsymbol{r}-\boldsymbol{r}_{h}(y))\\
		&\quad-b(z-z_{h}(y),\boldsymbol{\psi}_{I})-a(\boldsymbol{\psi}_{I},\boldsymbol{r}-\boldsymbol{r}_{h}(y))-b(\varphi_{I},\boldsymbol{r}-\boldsymbol{r}_{h}(y))\\
		&=-b(z-z_{h}(y),\boldsymbol{\psi}-\boldsymbol{\psi}_{I})-a(\boldsymbol{\psi}-\boldsymbol{\psi}_{I},\boldsymbol{r}-\boldsymbol{r}_{h}(y))-b(\varphi-\varphi_{I},\boldsymbol{r}-\boldsymbol{r}_{h}(u))\\
		&\quad +a(\boldsymbol{r}_{h}(y),\boldsymbol{\psi}_{I}-\boldsymbol{\Pi}^{0}_{0}\boldsymbol{\psi})+a_{h}(\boldsymbol{r}_{h}(y),\boldsymbol{\Pi}^{0}_{0}\boldsymbol{\psi}-\boldsymbol{\psi}_{I}),
	\end{aligned}
\end{equation*}
where $\boldsymbol{\psi_{I}}$ and $\varphi_{I}$ are interpolations. Following the interpolation error estimates of Lemma \ref{lem4.1}, we have
\begin{equation}\label{4.4}
	\begin{aligned}
		b(z-z_{h}(y),\boldsymbol{\psi}-\boldsymbol{\psi}_{I})&=b(z-z_{I},\boldsymbol{\psi}-\boldsymbol{\psi}_{I})+b(z_{I}-z_{h}(y),\boldsymbol{\psi}-\boldsymbol{\psi}_{I})\\
		&=b(z-z_{I},\boldsymbol{\psi}-\boldsymbol{\psi}_{I})+0\\
		&\le Ch\Vert z\Vert_{1,\Omega}\Vert\boldsymbol{\psi}\Vert_{1,\Omega}\\
		&\le Ch\Vert z\Vert_{1,\Omega}\Vert \xi\Vert_{0,\Omega}
	\end{aligned}
\end{equation}
and
\begin{equation}\label{4.5}
	\begin{aligned}
		a(\boldsymbol{\psi}-\boldsymbol{\psi}_{I},\boldsymbol{r}-\boldsymbol{r}_{h}(y))&\le\Vert\boldsymbol{\psi}-\boldsymbol{\psi}_{I}\Vert_{0,\Omega}\Vert\boldsymbol{r}-\boldsymbol{r}_{h}(y)\Vert_{0,\Omega}\\
		&\le Ch\Vert\boldsymbol{\psi}\Vert_{1,\Omega}\Vert y_{d}-\partial_{n_{A}}y\Vert_{\frac{1}{2},\Gamma}\\
		&\le Ch\Vert \xi\Vert_{0,\Omega}.
	\end{aligned}
\end{equation}
Moreover, from Cauchy-Schwarz inequality, \eqref{3.4} and Lemma \ref{lem4.1}, we have
\begin{equation}\label{4.6}
	\begin{aligned}
		a(\boldsymbol{r}_{h}(y),\boldsymbol{\psi}_{I}-\boldsymbol{\Pi}^{0}_{0}\boldsymbol{\psi})+a_{h}(\boldsymbol{r}_{h}(y),\boldsymbol{\Pi}^{0}_{0}\boldsymbol{\psi}-\boldsymbol{\psi}_{I})&\le C\Vert\boldsymbol{r}_{h}(y)\Vert_{0,\Omega}(\Vert \boldsymbol{\psi}_{I}-\boldsymbol{\psi}\Vert_{0,\Omega}+\Vert\boldsymbol{\psi}-\boldsymbol{\Pi}^{0}_{0}\boldsymbol{\psi}\Vert_{0,\Omega})\\
		&\le Ch\Vert\boldsymbol{\psi}\Vert_{1,\Omega}\Vert y_{d}-\partial_{n_{A}}y\Vert_{\frac{1}{2},\Gamma}\\
		&\le Ch\Vert \xi\Vert_{0,\Omega}.
	\end{aligned}
\end{equation}
Summing up, it follows from \eqref{4.4}-\eqref{4.6} that
\begin{equation*}
	\Vert z-z_{h}(y)\Vert_{0,\Omega}\le\underset{\xi\in L^{2}(\Omega),\ \xi\neq 0}{\text{sup}}\frac{(z-z_{h}(y),\xi)}{\Vert \xi\Vert_{0,\Omega}}\le Ch.
\end{equation*}
\end{proof}
\begin{lemma}\label{lem4.5}
	\textit{Given $E\in\mathcal{T}_{h}$, for any $\boldsymbol{v}_{h}\in\boldsymbol{V}_{h}^{E}$ the following estimate is valid}
	\begin{equation*}
		\Vert\boldsymbol{v}_{h}\cdot\boldsymbol{n}\Vert_{0,\partial E}\le Ch^{-\frac{1}{2}}\Vert\boldsymbol{v}_{h}\Vert_{0,E}.
	\end{equation*}
\end{lemma}
\begin{proof}
	From Lemma 2.3 of \cite{interpolat}, we have
	\begin{equation}\label{4.7}
		\Vert\boldsymbol{v}_{h}\cdot\boldsymbol{n}\Vert_{0,\partial E}\le Ch^{-\frac{1}{2}}\Vert\boldsymbol{v}_{h}\cdot\boldsymbol{n}\Vert_{-\frac{1}{2},\partial E}.
	\end{equation}
Since $\boldsymbol{v}_{h}\in H(\text{div},E)$ and $\text{div }\boldsymbol{v}_{h}\in\mathbb{P}_{k}(E)$, from Lemma 4.2 and 4.3 of \cite{trace} we can get
\begin{equation}\label{4.8}
	\Vert\boldsymbol{v}_{h}\cdot\boldsymbol{n}\Vert_{-\frac{1}{2},\partial E}\le C\Vert\boldsymbol{v}_{h}\Vert_{0,E}.
\end{equation}
Therefore, the Lemma can be followed by above inequalities.
\end{proof}

\begin{lemma}\label{lem4.6}
\textit{Let $\boldsymbol{p}$ be the solution of \eqref{2.6} and $\boldsymbol{p}_{h}(u)$ the solution of the auxiliary problems \eqref{4.1}, respectively. Then}
\begin{equation*}
	\begin{aligned}
		\Vert \boldsymbol{p}\cdot\boldsymbol{n}-\boldsymbol{p}_{h}(u)\cdot\boldsymbol{n}\Vert_{0,\Gamma}&\le Ch^{\frac{1}{2}}\Vert y\Vert_{2,\Omega},\\
		\Vert \boldsymbol{p}\cdot\boldsymbol{n}-\boldsymbol{p}_{h}(u)\cdot\boldsymbol{n}\Vert_{-\frac{1}{2},\Gamma}&\le Ch(\Vert y\Vert_{2,\Omega}+\Vert u\Vert_{1,\Omega}+\Vert f\Vert_{1,\Omega}).\\
	\end{aligned}
\end{equation*}
\end{lemma}
\begin{proof}
Let $\boldsymbol{p}_{I}=\boldsymbol{\Pi}_{h}\boldsymbol{p}\in\boldsymbol{V}_{h}$ be the interpolation of $\boldsymbol{p}$ defined in \eqref{3.2}, we have
\begin{equation*}
		\Vert \boldsymbol{p}\cdot\boldsymbol{n}-\boldsymbol{p}_{h}(u)\cdot\boldsymbol{n}\Vert_{0,\Gamma}\le \Vert \boldsymbol{p}\cdot\boldsymbol{n}-\boldsymbol{p}_{I}\cdot\boldsymbol{n}\Vert_{0,\Gamma}+\Vert\boldsymbol{p}_{I}\cdot\boldsymbol{n}-\boldsymbol{p}_{h}(u)\cdot\boldsymbol{n}\Vert_{0,\Gamma}.
\end{equation*}
For the first term of right-hand, from Lemma \ref{4.2}, we have
\begin{equation}\label{4.9}
	\begin{aligned}
		\Vert \boldsymbol{p}\cdot\boldsymbol{n}-\boldsymbol{p}_{I}\cdot\boldsymbol{n}\Vert_{0,\Gamma}&\le C\bigg(\sum_{e\in\mathcal{S}_{h}(\Gamma)}\Vert\boldsymbol{p}\cdot\boldsymbol{n}_{e}-\boldsymbol{p}_{I}\cdot\boldsymbol{n}_{e}\Vert_{0,e}^{2}\bigg)^{\frac{1}{2}}\\
		&\le Ch^{\frac{1}{2}}\bigg(\sum_{\mathcal{S}_{h}(E)\cap\mathcal{S}_{h}(\Gamma)\neq\emptyset}\Vert y\Vert_{2,E}^{2}\bigg)^{\frac{1}{2}}\\
		&\le Ch^{\frac{1}{2}}\Vert y\Vert_{2,\Omega}.
	\end{aligned}
\end{equation}
Futhermore, by using Lemma \ref{4.5}, we prove that
\begin{equation}\label{4.10}
	\begin{aligned}
		\Vert\boldsymbol{p}_{I}\cdot\boldsymbol{n}-\boldsymbol{p}_{h}(u)\cdot\boldsymbol{n}\Vert_{0,\Gamma}&\le C\bigg(\sum_{e\in\mathcal{S}_{h}(\Gamma)}\Vert\boldsymbol{p}_{I}\cdot\boldsymbol{n}_{e}-\boldsymbol{p}_{h}(u)\cdot\boldsymbol{n}_{e}\Vert_{0,e}^{2}\bigg)^{\frac{1}{2}}\\
		&\le Ch^{-\frac{1}{2}}\bigg(\sum_{\mathcal{S}_{h}(E)\cap\mathcal{S}_{h}(\Gamma)\neq\emptyset}\Vert\boldsymbol{p}_{I}-\boldsymbol{p}_{h}(u)\Vert_{0,E}^{2}\bigg)^{\frac{1}{2}}\\
		&\le Ch^{-\frac{1}{2}}\Vert\boldsymbol{p}_{I}-\boldsymbol{p}_{h}(u)\Vert_{0,\Omega}\\
		&\le Ch^{-\frac{1}{2}}(\Vert\boldsymbol{p}_{I}-\boldsymbol{p}\Vert_{0,\Omega}+\Vert\boldsymbol{p}-\boldsymbol{p}_{h}(u)\Vert_{0,\Omega})\\
		&\le Ch^{\frac{1}{2}}\Vert y\Vert_{2,\Omega}.
	\end{aligned}
\end{equation}
From \eqref{4.9}-\eqref{4.10}, we conclude that
\begin{equation}\label{4.11}
	\Vert \boldsymbol{p}\cdot\boldsymbol{n}-\boldsymbol{p}_{h}(u)\cdot\boldsymbol{n}\Vert_{0,\Gamma}\le Ch^{\frac{1}{2}}\Vert y\Vert_{2,\Omega}.
\end{equation}
For the estimate of $\Vert\boldsymbol{p}\cdot\boldsymbol{n}-\boldsymbol{p}_{h}(u)\cdot\boldsymbol{n}\Vert_{-\frac{1}{2},\Gamma}$, it follows from Lemma \ref{lem2.1} and \ref{lem4.3}
\begin{equation}\label{4.12}
	\begin{aligned}
		\Vert\boldsymbol{p}\cdot\boldsymbol{n}-\boldsymbol{p}_{h}(u)\cdot\boldsymbol{n}\Vert_{-\frac{1}{2},\Gamma}&\le\Vert\boldsymbol{p}-\boldsymbol{p}_{h}(u)\Vert_{H(\text{div}, \Omega)}\\
		&\le Ch(\Vert y\Vert_{2,\Omega}+\Vert u\Vert_{1,\Omega}+\Vert f\Vert_{1,\Omega}).
	\end{aligned}
\end{equation}
	This completes the proof.
	\end{proof}
Next, we consider the approximation of the objective functional $J$. Let
\begin{equation*}
J_{h}'(u)(w)=(\gamma u+z_{h}(u),w).
\end{equation*}
where $z_{h}(u)$ is the solution of the auxiliary problem \eqref{4.1}. Then, we have the following lemma.
	\begin{lemma}\label{lem4.7}
		\textit{Let $J_{h}$ be defined in \eqref{3.8}, then we have}
		\begin{equation*}
			J_{h}'(v)(v-u)-J_{h}'(u)(v-u)= \gamma\Vert v-u\Vert_{0,\Omega}^{2}+\Vert\boldsymbol{p}_{h}(v)\cdot\boldsymbol{n}-\boldsymbol{p}_{h}(u)\cdot\boldsymbol{n} \Vert_{0,\Gamma}^{2}.
		\end{equation*}
	\end{lemma}
	\begin{proof}
		Note that
			\begin{equation*}
				J_{h}'(v)(v-u)=(\gamma v+z_{h}(v),v-u).
			\end{equation*}
		Then, we have
\begin{equation}\label{4.13}
	\begin{aligned}
		J_{h}'(v)(v-u)-J_{h}'(u)(v-u)&=(\gamma v+z_{h}(v),v-u)-(\gamma u+z_{h}(u),v-u)\\
		&=\gamma\Vert v-u\Vert_{0,\Omega}^{2}+(z_{h}(v)-z_{h}(u),v-u).
	\end{aligned}
\end{equation}
From \eqref{4.1}, we obtain
\begin{equation}\label{4.14}
	\begin{aligned}
		(z_{h}(v)-z_{h}(u),v-u)&=-b(z_{h}(v)-z_{h}(u),\boldsymbol{p}_{h}(v)-\boldsymbol{p}_{h}(u))\\
		&=\langle\boldsymbol{p}_{h}(v)\cdot\boldsymbol{n}-\boldsymbol{p}_{h}(u)\cdot\boldsymbol{n},\boldsymbol{p}_{h}(v)\cdot\boldsymbol{n}-\boldsymbol{p}_{h}(u)\cdot\boldsymbol{n}\rangle+a_{h}(\boldsymbol{r}_{h}(v)-\boldsymbol{r}_{h}(u),\boldsymbol{p}_{h}(v)-\boldsymbol{p}_{h}(u))\\
		&=\Vert(\boldsymbol{p}_{h}(v)-\boldsymbol{p}_{h}(u))\cdot\boldsymbol{n} \Vert_{0,\Gamma}^{2}-b(y_{h}(v)-y_{h}(u),\boldsymbol{r}_{h}(v)-\boldsymbol{r}_{h}(u))\\
		&=\Vert(\boldsymbol{p}_{h}(v)-\boldsymbol{p}_{h}(u))\cdot\boldsymbol{n} \Vert_{0,\Gamma}^{2}.
	\end{aligned}
\end{equation}
Therefore, the Lemma can be followed by \eqref{4.13}-\eqref{4.14}.
	\end{proof}

\begin{theorem}
	\textit{Let $(\boldsymbol{p},y,\boldsymbol{r},z,u)\in\boldsymbol{V}\times W\times\boldsymbol{V}\times W\times U_{ad}$ be the solution of \eqref{2.6}, and $(\boldsymbol{p}_{h},y_{h},\boldsymbol{r}_{h},z_{h},u_{h})\in\boldsymbol{V}_{h}\times W_{h}\times\boldsymbol{V}_{h}\times W_{h}\times U_{ad}$ be the solution of the discretized problem \eqref{3.10}. Then we have}
	\begin{equation*}
		\Vert y-y_{h}\Vert_{0,\Omega}+ \Vert z-z_{h}\Vert_{0,\Omega}+\Vert u-u_{h}\Vert_{0,\Omega}\le Ch.
	\end{equation*}
\end{theorem}
\begin{proof}
According to inf-sup condition of Lemma \ref{lem3.3}, we have
\begin{equation}\label{4.15}
	\begin{aligned}
		\beta\Vert y_{h}(u)-y_{h}\Vert_{0,\Omega}&\le\underset{\boldsymbol{v}_{h}\in\boldsymbol{V}_{h}}{\text{sup}}\frac{b(y_{h}(u)-y_{h},\boldsymbol{v}_{h})}{\Vert\boldsymbol{v}_{h}\Vert_{0,\Omega}}\\
		&=\underset{\boldsymbol{v}_{h}\in\boldsymbol{V}_{h}}{\text{sup}}\frac{a_{h}(\boldsymbol{p}_{h}-\boldsymbol{p}_{h}(u),\boldsymbol{v}_{h})  }{\Vert\boldsymbol{v}_{h}\Vert_{0,\Omega}}\\
		&\le \Vert\boldsymbol{p}_{h}-\boldsymbol{p}_{h}(u)\Vert_{0,\Omega}.
	\end{aligned}
\end{equation}
From Lemma \ref{lem3.2}, we have
\begin{equation}\label{4.16}
	\begin{aligned}
		\alpha\Vert\boldsymbol{p}_{h}-\boldsymbol{p}_{h}(u)\Vert_{0,\Omega}^{2}&\le \vert a_{h}(\boldsymbol{p}_{h}-\boldsymbol{p}_{h}(u),\boldsymbol{p}_{h}-\boldsymbol{p}_{h}(u))\vert\\
		&= \vert b(y_{h}(u)-y_{h},\boldsymbol{p}_{h}-\boldsymbol{p}_{h}(u))\vert\\
		&= \vert (u-u_{h},y_{h}(u)-y_{h})\vert\\
		&\le \Vert u-u_{h}\Vert_{0,\Omega}\Vert y_{h}(u)-y_{h}\Vert_{0,\Omega}.
	\end{aligned}
\end{equation}
Combining above inequalities, we obtain
\begin{equation}\label{4.17}
	\Vert y_{h}(u)-y_{h}\Vert_{0,\Omega}\le C\Vert u-u_{h}\Vert_{0,\Omega}.
\end{equation}
From Lemma \ref{lem4.7}, we have
\begin{equation}\label{4.18}
	\begin{aligned}
	\gamma\Vert u-u_{h}\Vert_{0,\Omega}^{2}+\Vert\boldsymbol{p}_{h}(u)\cdot\boldsymbol{n}-\boldsymbol{p}_{h}\cdot\boldsymbol{n} \Vert_{0,\Gamma}^{2}&=J_{h}'(u)(u-u_{h})-J_{h}'(u_{h})(u-u_{h})\\
	&=(\gamma u+z_{h}(u),u-u_{h})-(\gamma u_{h}+z_{h},u-u_{h})\\
	&=(\gamma u+z,u-u_{h})-(z-z_{h}(u),u-u_{h})-(\gamma u_{h}+z_{h},u-u_{h}).\\
	\end{aligned}
\end{equation}
Since the control variable adopts the variational discretization, we have
\begin{equation*}
	(\gamma u+z,u-u_{h})\le 0
\end{equation*}
and
\begin{equation*}
	(\gamma u_{h}+z_{h},u-u_{h})\le 0.
\end{equation*}
Then, from \eqref{4.18} and Young inequality, we conclude that
\begin{equation}\label{4.19}
	\begin{aligned}
		\gamma\Vert u-u_{h}\Vert_{0,\Omega}^{2}+\Vert\boldsymbol{p}_{h}(u)\cdot\boldsymbol{n}-\boldsymbol{p}_{h}\cdot\boldsymbol{n} \Vert_{0,\Gamma}^{2}&\le (z_{h}(u)-z,u-u_{h})\\
		&=(z_{h}(u)-z_{h}(y),u-u_{h})+(z_{h}(y)-z,u-u_{h})\\
		&\le C(\varepsilon,\delta)\Vert z_{h}(u)-z_{h}(y)\Vert_{0,\Omega}^2+C(\varepsilon,\delta)\Vert z_{h}(y)-z\Vert_{0,\Omega}^2\\
		&\quad+\varepsilon\Vert u-u_{h}\Vert_{0,\Omega},
	\end{aligned}
\end{equation}
where $\varepsilon,\delta$ denote arbitrary small positive numbers.

For first term of the right-hand of \eqref{4.19}, we use a duality argument. Let $(\boldsymbol{\sigma},\zeta)\in H(\mbox{div}, \Omega)\times L^{2}(\Omega)$ satisfy
\begin{equation*}
	\left\{
	\begin{aligned}
		a(\boldsymbol{\sigma},\boldsymbol{v})+b(\zeta,\boldsymbol{v})&=0  &&\forall \boldsymbol{v}\in\boldsymbol{V},\\
		b(w,\boldsymbol{\sigma})&=(z_{h}(u)-z_{h}(y),w) &&\forall w\in W.\\
	\end{aligned}
    \right.
\end{equation*}
From the regularity theorey, we have
\begin{equation}\label{4.20}
	\Vert\zeta\Vert_{2,\Omega}+\Vert\boldsymbol{\sigma}\Vert_{1,\Omega}\le C\Vert z_{h}(u)-z_{h}(y)\Vert_{0,\Omega}.
\end{equation}
Let $(\boldsymbol{\sigma}_{h},\zeta_{h})\in\boldsymbol{V}_{h}\times W_{h}$ be the mixed virtual element approximation of $(\boldsymbol{\sigma},\zeta)$ such that
\begin{equation}\label{4.21}
	\left\{
	\begin{aligned}
		a_{h}(\boldsymbol{\sigma}_{h},\boldsymbol{v}_{h})+b(\zeta_{h},\boldsymbol{v}_{h})&=0 &&\forall\boldsymbol{v}_{h}\in\boldsymbol{V}_{h},\qquad\qquad\\
		b(w_{h},\boldsymbol{\sigma}_{h})&=(z_{h}(u)-z_{h}(y),w_{h}) &&\forall w_{h}\in W_{h}.
	\end{aligned}
\right.
\end{equation}
Similar to Lemma \ref{lem4.6}, we can prove
\begin{equation}\label{4.22}
	\Vert(\boldsymbol{\sigma}-\boldsymbol{\sigma}_{h})\cdot \boldsymbol{n}\Vert_{0,\Gamma}\le Ch^{\frac{1}{2}}\Vert\zeta\Vert_{2,\Omega}.
\end{equation}
From Lemma \ref{lem4.6}, \eqref{4.1} and \eqref{4.21}, we have
\begin{equation}\label{4.23}
	\begin{aligned}
		\Vert z_{h}(u)-z_{h}(y)\Vert_{0,\Omega}^{2}&=b(z_{h}(u)-z_{h}(y),\boldsymbol{\sigma}_{h})\\
		&=-a_{h}(\boldsymbol{r}_{h}(u)-\boldsymbol{r}_{h}(y),\boldsymbol{\sigma}_{h})+\langle(\boldsymbol{p}-\boldsymbol{p}_{h}(u))\cdot\boldsymbol{n},\boldsymbol{\sigma}_{h}\cdot \boldsymbol{n}\rangle\\
		&=b(\zeta_{h},\boldsymbol{r}_{h}(u)-\boldsymbol{r}_{h}(y))+\langle(\boldsymbol{p}-\boldsymbol{p}_{h}(u))\cdot\boldsymbol{n},\boldsymbol{\sigma}_{h}\cdot \boldsymbol{n}\rangle\\
		&=\langle(\boldsymbol{p}-\boldsymbol{p}_{h}(u))\cdot\boldsymbol{n},(\boldsymbol{\sigma}_{h}-\boldsymbol{\sigma})\cdot \boldsymbol{n}\rangle+\langle(\boldsymbol{p}-\boldsymbol{p}_{h}(u))\cdot\boldsymbol{n},\boldsymbol{\sigma}\cdot \boldsymbol{n}\rangle\\
		&\le \Vert(\boldsymbol{p}-\boldsymbol{p}_{h}(u))\cdot\boldsymbol{n}\Vert_{0,\Gamma}\Vert(\boldsymbol{\sigma}_{h}-\boldsymbol{\sigma})\cdot \boldsymbol{n}\Vert_{0,\Gamma}+\Vert(\boldsymbol{p}-\boldsymbol{p}_{h}(u))\cdot\boldsymbol{n}\Vert_{-\frac{1}{2},\Gamma}\Vert\boldsymbol{\sigma}\cdot \boldsymbol{n}\Vert_{\frac{1}{2},\Gamma}\\
		&\le Ch\Vert y\Vert_{2,\Omega}\Vert\zeta\Vert_{2,\Omega}+Ch\Vert\boldsymbol{\sigma}\Vert_{1,\Omega}.
	\end{aligned}
\end{equation}
Therefore, it follows from \eqref{4.20} and \eqref{4.23} that                    
\begin{equation}\label{4.24}
	\Vert z_{h}(u)-z_{h}(y)\Vert_{0,\Omega}\le Ch.
\end{equation}
Combining the results of Lemma \ref{lem4.3}, \eqref{4.19} and \eqref{4.24} and setting $\varepsilon,\delta$ small enough, we have
\begin{equation}\label{4.25}
	\gamma\Vert u-u_{h}\Vert_{0,\Omega}^{2}+\Vert\boldsymbol{p}_{h}(u)\cdot\boldsymbol{n}-\boldsymbol{p}_{h}\cdot\boldsymbol{n} \Vert_{0,\Gamma}^{2}\le Ch^{2}.
\end{equation}
From Lemma \ref{lem4.3} and \eqref{4.17}, we have
\begin{equation}\label{4.26}
	\Vert y-y_{h}\Vert_{0,\Omega}\le Ch.
\end{equation}
To estimate $\Vert z_{h}(u)-z_{h}\Vert_{0,\Omega}$, we still use a duality argument. Let $(\boldsymbol{\psi},\varphi)\in H(\mbox{div}, \Omega)\times L^{2}(\Omega)$ satisfy
\begin{equation}\label{4.27}
	\left\{
	\begin{aligned}
		a(\boldsymbol{\psi},\boldsymbol{v})+b(\varphi,\boldsymbol{v})&=0  &&\forall \boldsymbol{v}\in\boldsymbol{V},\\
		b(w,\boldsymbol{\psi})&=-(z_{h}(u)-z_{h},w) &&\forall w\in W.\\
	\end{aligned}
\right.
\end{equation}
Then the mixed virtual element scheme of \eqref{4.27} is as follows
\begin{equation}\label{4.28}
	\left\{
	\begin{aligned}
		a(\boldsymbol{\psi}_{h},\boldsymbol{v}_{h})+b(\varphi_{h},\boldsymbol{v}_{h})&=0  &&\forall \boldsymbol{v}_{h}\in\boldsymbol{V}_{h},\\
		b(w,\boldsymbol{\psi}_{h})&=-(z_{h}(u)-z_{h},w_{h}) &&\forall w_{h}\in W_{h}.\\
	\end{aligned}
	\right.
\end{equation}
Similarly to estimate $\Vert z_{h}(u)-z_{h}(y)\Vert_{0,\Omega}$, we have
\begin{equation}\label{4.29}
	\begin{aligned}
		\Vert z_{h}(u)-z_{h}\Vert_{0,\Omega}^{2}&=b(z_{h}(u)-z_{h},\boldsymbol{\psi}_{h})\\
		&=-a_{h}(\boldsymbol{r}_{h}(u)-\boldsymbol{r}_{h},\boldsymbol{\psi}_{h})+\langle(\boldsymbol{p}_{h}-\boldsymbol{p}_{h}(u))\cdot\boldsymbol{n},\boldsymbol{\psi}_{h}\cdot \boldsymbol{n}\rangle\\
		&=b(\varphi_{h},\boldsymbol{r}_{h}(u)-\boldsymbol{r}_{h})+\langle(\boldsymbol{p}_{h}-\boldsymbol{p}_{h}(u))\cdot\boldsymbol{n},\boldsymbol{\psi}_{h}\cdot \boldsymbol{n}\rangle\\
		&=0+\langle(\boldsymbol{p}_{h}-\boldsymbol{p}_{h}(u))\cdot\boldsymbol{n},\boldsymbol{\psi}_{h}\cdot \boldsymbol{n}\rangle\\
		&\le \Vert(\boldsymbol{p}_{h}-\boldsymbol{p}_{h}(u))\cdot\boldsymbol{n}\Vert_{0,\Gamma}(\Vert(\boldsymbol{\psi}-\boldsymbol{\psi}_{h})\cdot \boldsymbol{n}\Vert_{0,\Gamma}+\Vert\boldsymbol{\psi}\cdot \boldsymbol{n}\Vert_{0,\Gamma})\\
		&\le Ch\Vert\varphi\Vert_{2,\Omega}.
	\end{aligned}
\end{equation}
Therefore,
\begin{equation}\label{4.30}
	\Vert z_{h}(u)-z_{h}\Vert_{0,\Omega}\le Ch.
\end{equation}
Thus, it follows by Lemma \ref{lem4.4}, \eqref{4.23} and \eqref{4.30} that
\begin{equation}\label{4.31}
	\Vert z-z_{h}\Vert_{0,\Omega}\le Ch.
\end{equation}
Summing up \eqref{4.25}, \eqref{4.26} and \eqref{4.31}, the Theorem follows.
\end{proof}

\section{Numerical experiment}
In this section, we will verify our theoretical analysis by numerical experiment using mixed virtual element method. Numerical experiments are carried out on square, random and nonconvex polygon meshes, which are shown in Figure \ref{fig1}, respectively. Due to the reduced regularity of the optimal control problem, we use the lowest order space to approximate the state variable $y$ and the adjoint variable $z$. The discrete control variable $u_{h}$ is calculated by using the variational discretization. We use the fixed point iteration and the code is based on Matlab with the mVEM package (see \cite{code}).

\begin{figure}[H]
	\begin{center}
	    \includegraphics[width=0.3\linewidth,trim=70 30 60 20,clip]{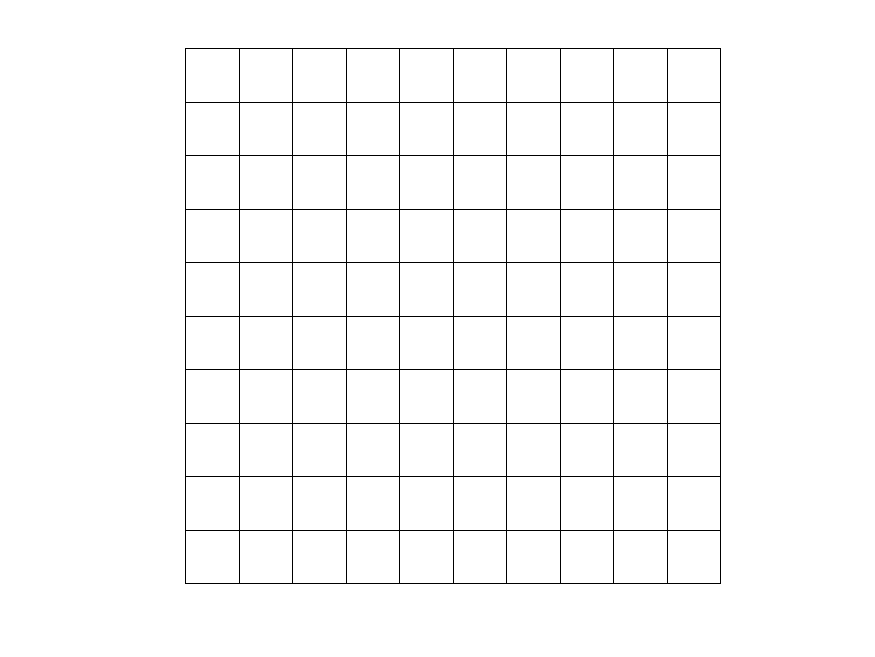}
		\includegraphics[width=0.3\linewidth,trim=70 30 60 20,clip]{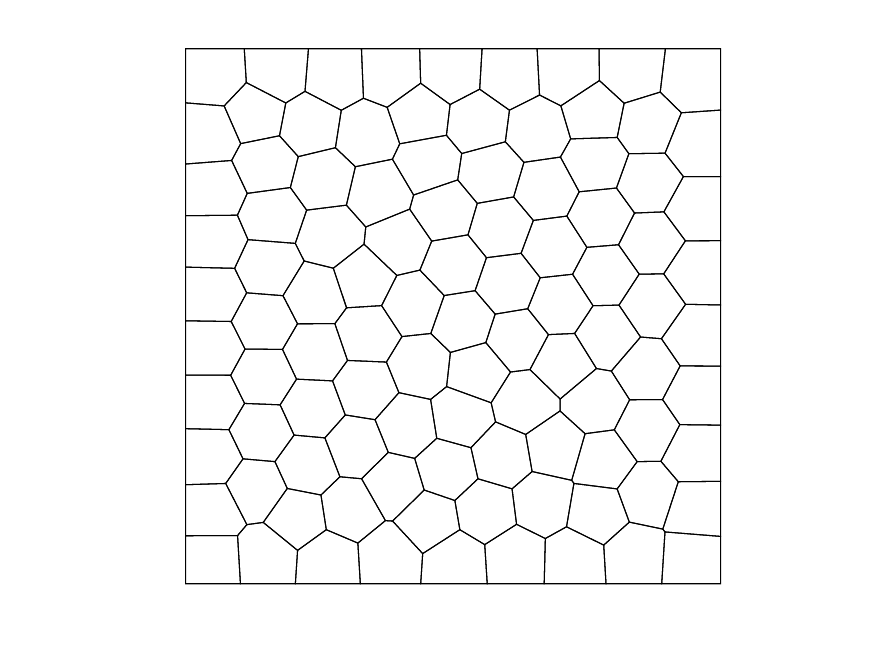}
		\includegraphics[width=0.3\linewidth,trim=70 30 60 20,clip]{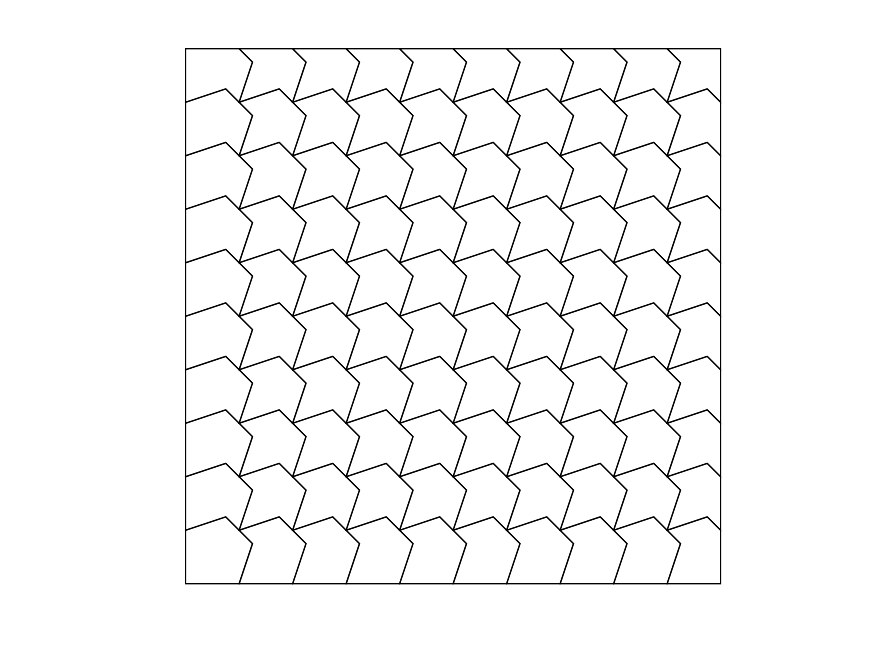}
		\caption{Different meshes}\label{fig1}		
	\end{center}
	\vspace{-1em}
\end{figure}
\noindent\textbf{Example 5.1.} The first example is defined on the unit square $\Omega = [0,1]\times[0,1]$. The true solution is chosen as
	\begin{equation*}
		\begin{aligned}
			y&=\mbox{sin}(\pi x_{1})\mbox{sin}(\pi x_{2})+1,\\
			z&=(x_{1}^{2}-x_{1})+(x_{2}-x_{2}^{2}),\\
			u&=P_{U_{ad}}\{-\frac{1}{\gamma}z\},\quad \gamma=1.
		\end{aligned}
	\end{equation*}
	We set $A$ as identity matrix. Consequently, the data should be calculated as
	\begin{equation*}
		\begin{aligned}
		f&=-\text{div}A\nabla y-u,\\
		y_{d}\vert_{\Gamma}&=\frac{\partial y}{\partial n_{A}}+z
	\end{aligned}
	\end{equation*}
and
\begin{equation*}
	U_{ad}:=\{u\in L^{2}(\Omega),\ 0\le u\le 0.5\ \mbox{a.e. in }\Omega \}.
\end{equation*}

For the Example 5.1, the numerical results about the variables $y,z,u$ on three different meshes are displayed in Table 1-3. We have $O(h)$ orders for $u$, $y$ and $z$ by using the mixed virtual element method, which is the optimal convergence order that the lowest order space can deliver.
\begin{figure}[H]
	\begin{center}
		\includegraphics[width=0.3\linewidth]{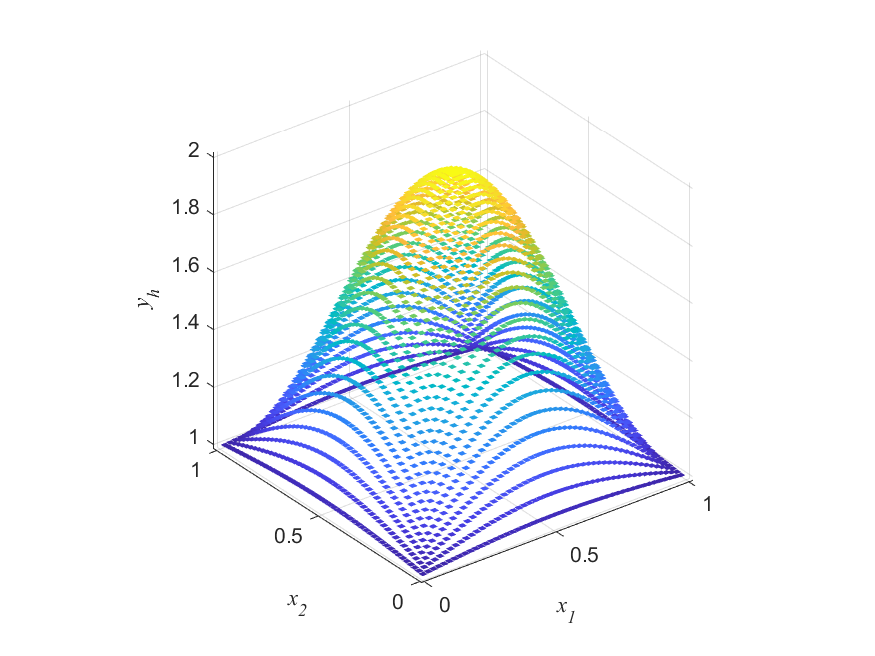}
		\includegraphics[width=0.3\linewidth]{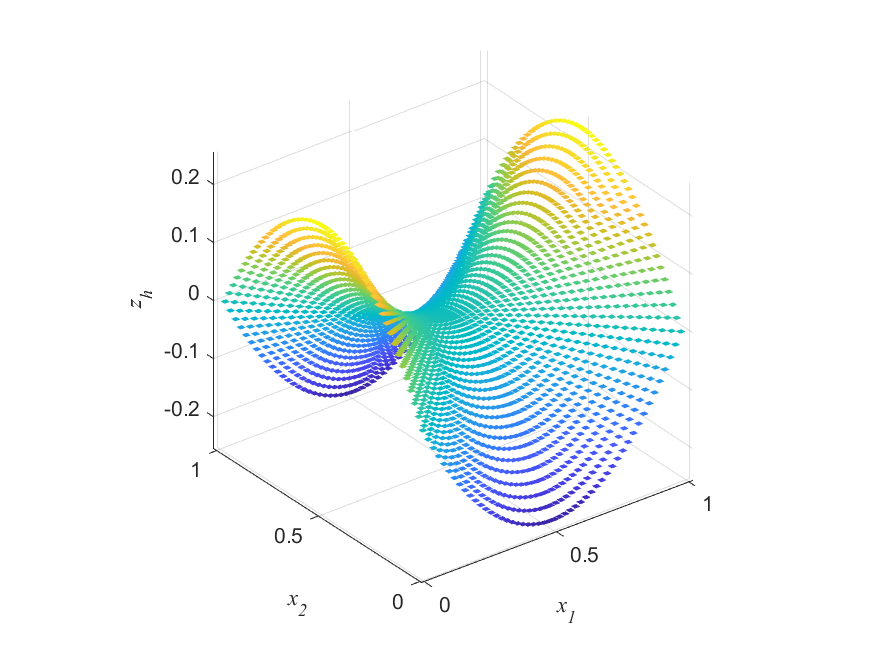}
		\includegraphics[width=0.3\linewidth]{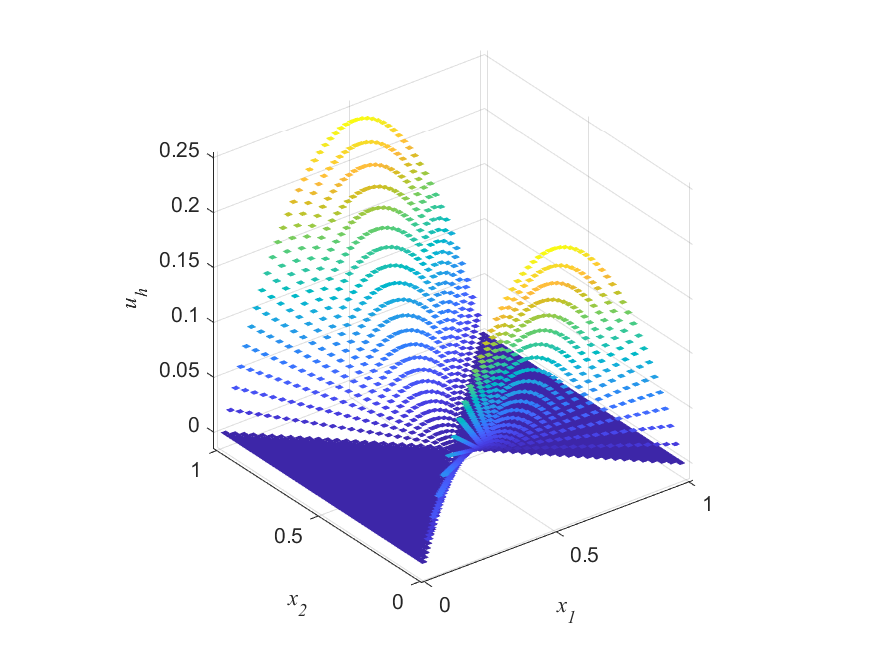}
		\caption{Different meshes}\label{fig21}		
	\end{center}
	\vspace{-1em}
\end{figure}
\begin{figure}[H]
	\begin{center}
		\includegraphics[width=0.3\linewidth]{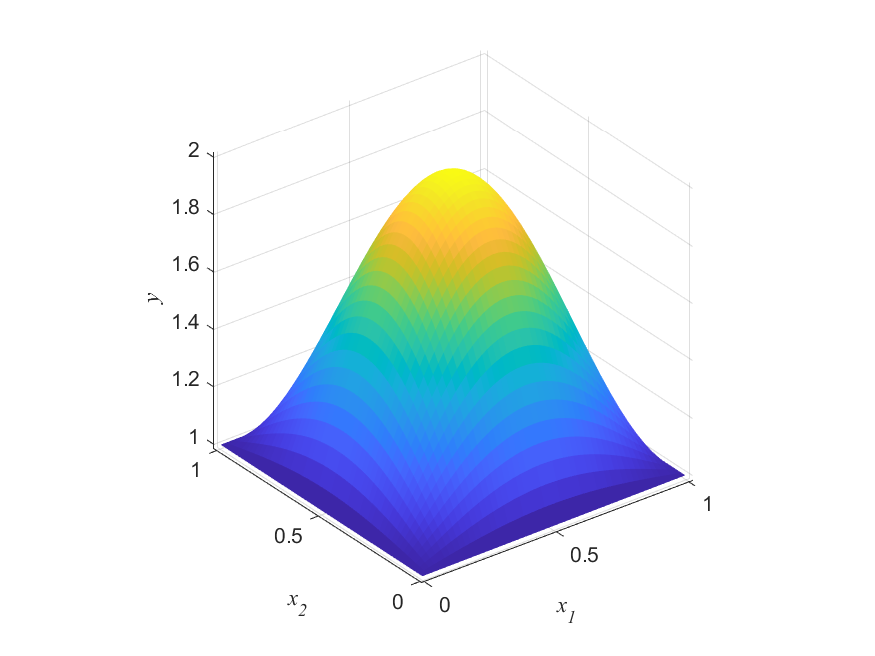}
		\includegraphics[width=0.3\linewidth]{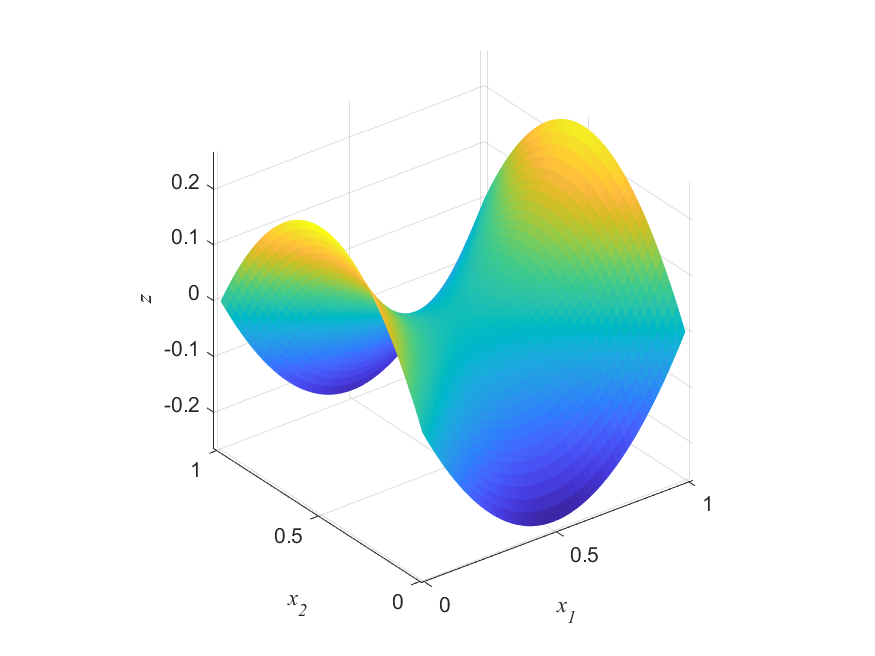}
		\includegraphics[width=0.3\linewidth]{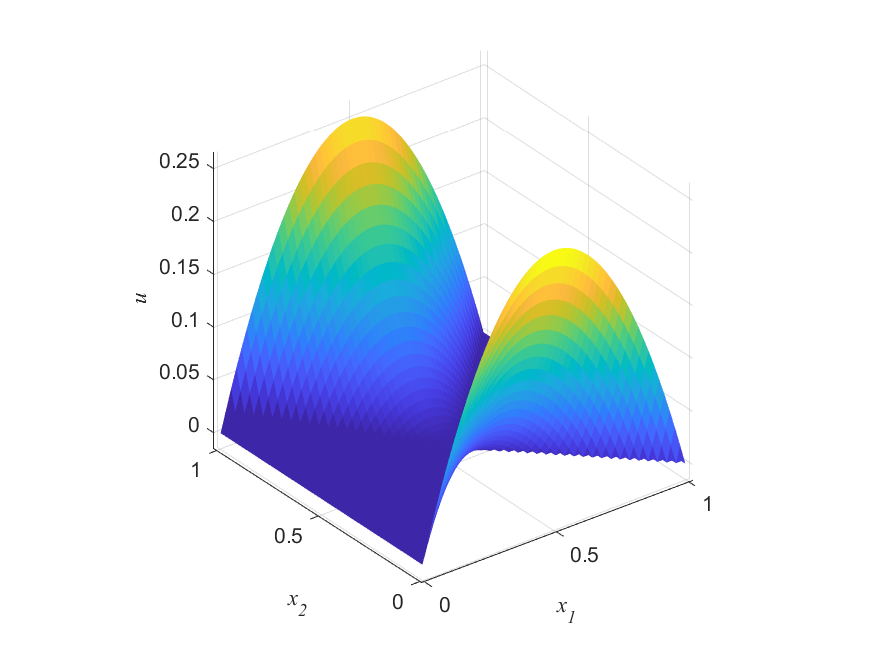}
		\caption{Different meshes}\label{fig22}		
	\end{center}
	\vspace{-1em}
\end{figure}
\begin{figure}[H]
	\begin{center}
		\includegraphics[width=0.3\linewidth]{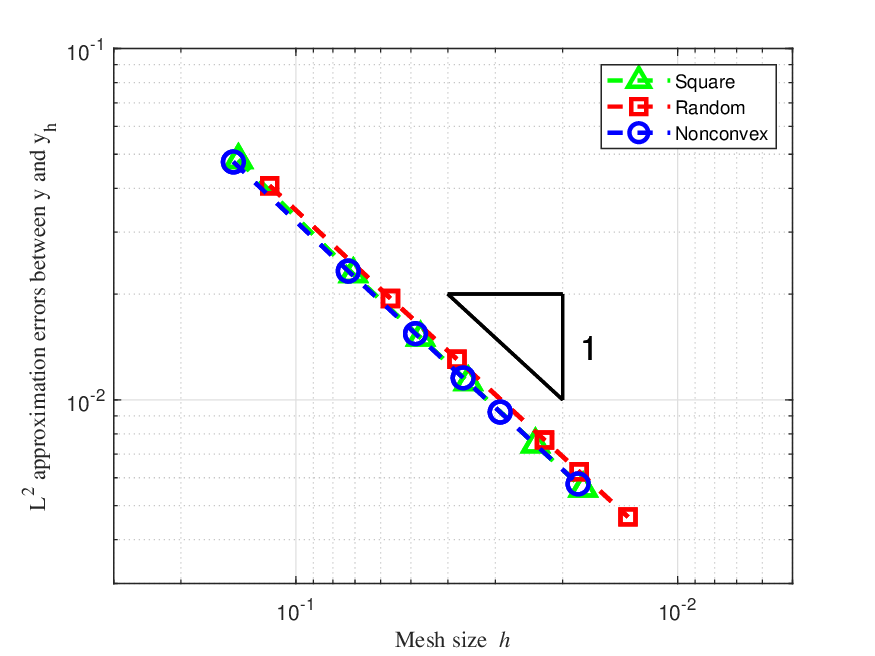}
		\includegraphics[width=0.3\linewidth]{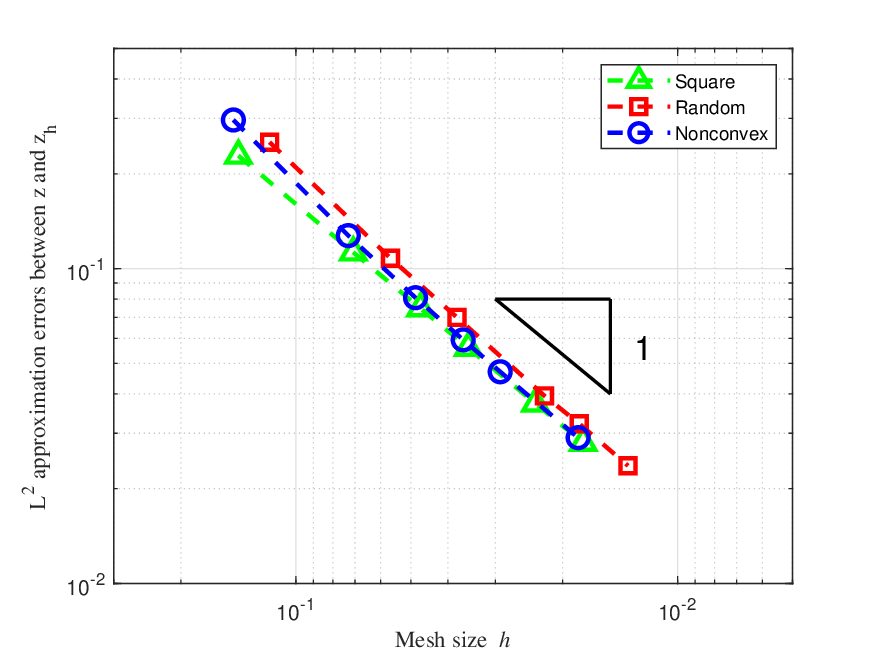}
		\includegraphics[width=0.3\linewidth]{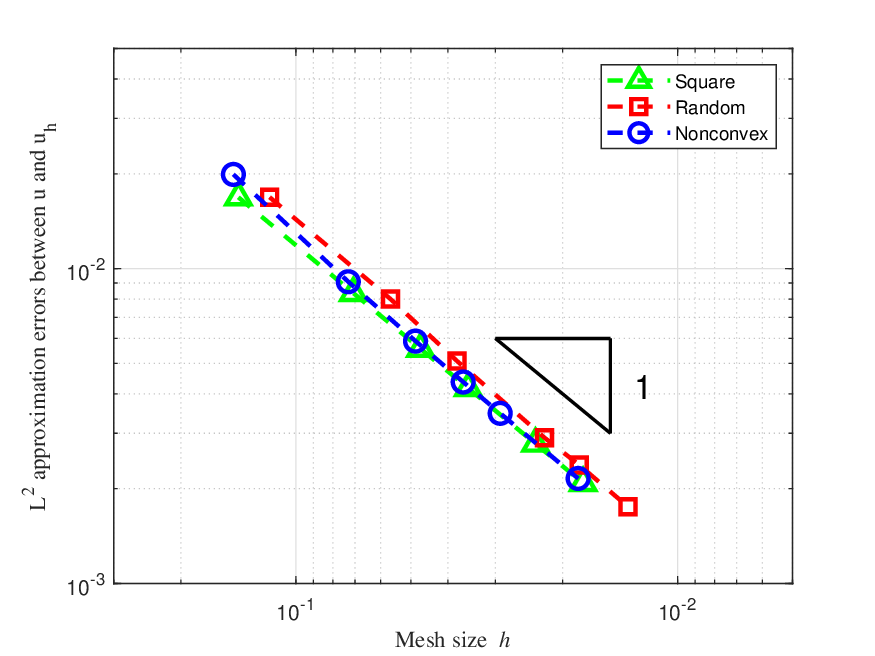}
		\caption{Different meshes}\label{fig23}		
	\end{center}
	\vspace{-1em}
\end{figure}

\noindent\textbf{Example 5.2. }The second example is also defined on the unit square $\Omega = [0,1]\times[0,1]$ (see \cite{4}). The data is chosen as
\begin{equation*}
	\begin{aligned}
		f&=\frac{1}{(x_{1}^{2}+x_{2}^{2})^{\frac{2}{5}}},\ y\vert_{\Gamma}=0,\\
		y_{d}&=0.
	\end{aligned}
\end{equation*}
We set $A$ as identity matrix, $\gamma=1$, and the constrained control set is chosen as
\begin{equation*}
	U_{ad}:=\{u\in L^{2}(\Omega),\ -0.5\le u\le -0.1\ \mbox{a.e. in }\Omega \}.
\end{equation*}

Note that we have no explicit expression of the exact solution, so we have solved it numerically with 250 000 elements on square mesh, and then we used this solution for comparison with other solutions on the meshes with bigger $h$. In this example, $f\in L^{2}(\Omega)$ but $f\notin L^{\frac{5}{2}}(\Omega)$. The numerical results are presented in Table 4-6 when $f$ has less regularity. We observe the convergence order is $O(h)$, which is the optimal convergence order.

\begin{figure}[H]
	\begin{center}
		\includegraphics[width=0.3\linewidth]{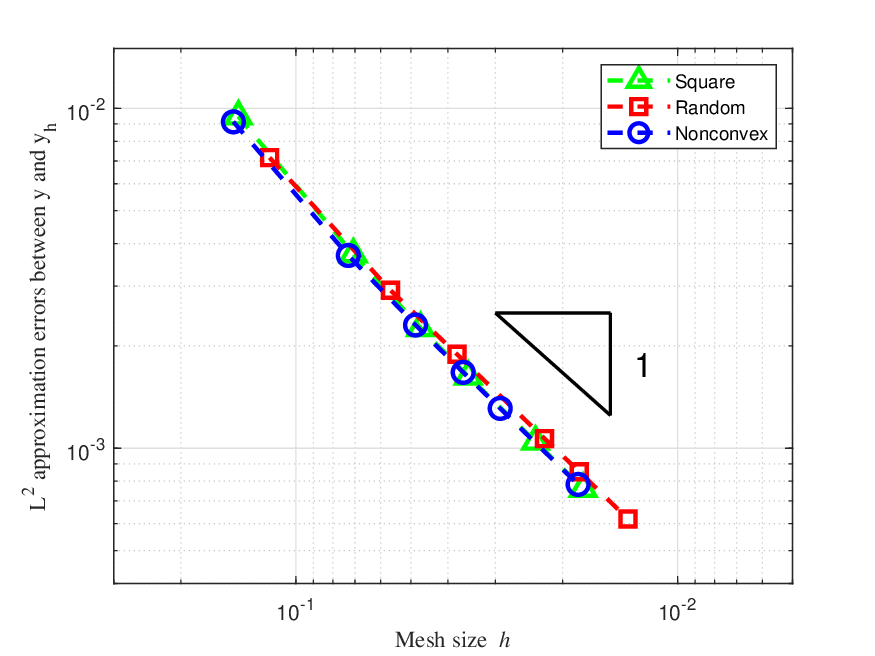}
		\includegraphics[width=0.3\linewidth]{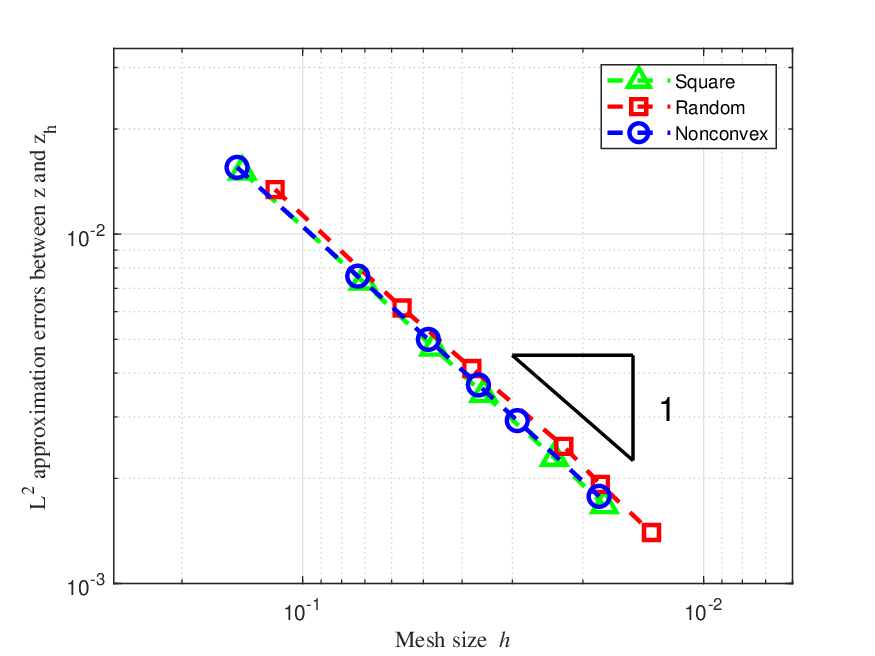}
		\includegraphics[width=0.3\linewidth]{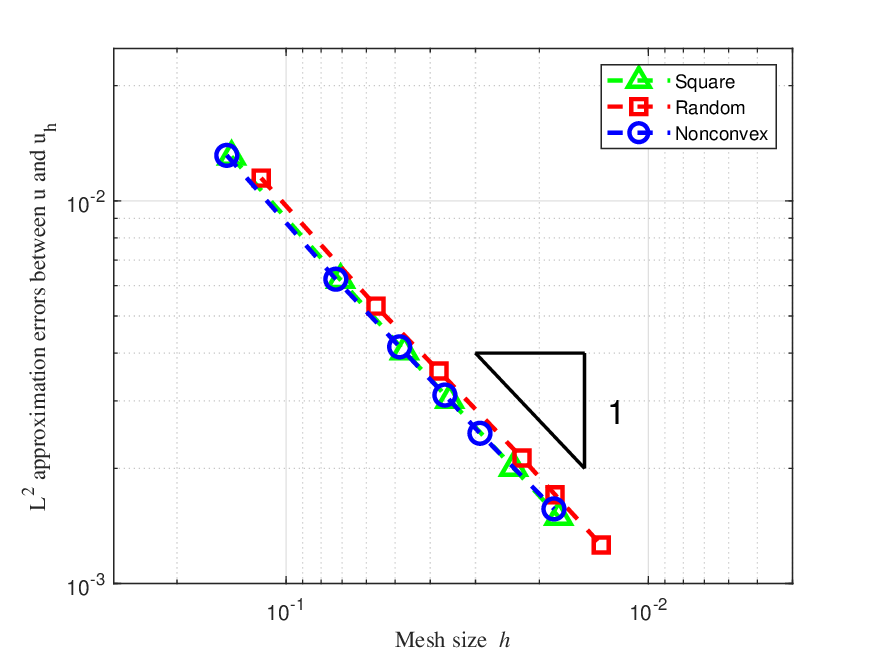}
		\caption{Different meshes}\label{fig33}		
	\end{center}
	\vspace{-1em}
\end{figure}

\begin{table}[H]
	\centering
	\caption{Convergence results for the errors of $y$, $z$, $u$ for Example 5.2 on square meshes}
	\begin{tabular}{ccccccc}
		\toprule %[2pt]   
		$h$ & 0.1414  & 0.0707 &0.0471 &0.0353 &0.0236& 0.0176 \\
		\midrule %[2pt]  
		$\Vert y-y_{h}\Vert_{0,\Omega}$ & 9.4677e-03 & 3.7207e-03 & 2.2582e-03&1.6256e-03 &1.0456e-03&7.5993e-04 \\
		Order& $\backslash$ & 1.3474 &1.2315 &1.1426 &1.0884&1.1093  \\
		$\Vert z-z_{h}\Vert_{0,\Omega}$ & 1.5097e-02 &7.3138e-03 &4.7321e-03  &3.4892e-03 &2.2902e-03&1.6752e-03  \\
		Order& $\backslash$ & 1.0456 &1.0738 &1.0592 &1.0383&1.0871  \\
		$\Vert u-u_{h}\Vert_{0,\Omega}$ & 1.3039e-02 & 6.2118e-03 &4.0437e-03  &3.0181e-03 &2.0064e-03&1.4934e-03  \\
		Order& $\backslash$ & 1.0698 &1.0588 &1.0169 &1.0069&1.0265 \\
		\bottomrule %[2pt]    
	\end{tabular}
\end{table}

\begin{table}[H]
	\centering
	\caption{Convergence results for the errors of $y$, $z$, $u$ for Example 5.2 on random meshes}
	\begin{tabular}{ccccccc}
		\toprule %[2pt]   
		$h$ &0.1171 & 0.0565 &0.0378 & 0.0223&0.0181& 0.0135 \\
		\midrule %[2pt]  
		$\Vert y-y_{h}\Vert_{0,\Omega}$ & 7.1603e-03 & 2.9152e-03 & 1.8907e-03&1.0674e-03 &8.5323e-04&6.1902e-04\\
		Order& $\backslash$ & 1.2322 &1.0811 & 1.0819&1.0678&1.0932 \\
		$\Vert z-z_{h}\Vert_{0,\Omega}$ & 1.3415e-02 &6.1491e-03 &4.1275e-03  &2.4687e-03 &1.9219e-03&1.3990e-03 \\
		Order& $\backslash$ & 1.0697 &0.9953 &0.9727 &1.1934&1.0818 \\
		$\Vert u-u_{h}\Vert_{0,\Omega}$ & 1.1461e-02 & 5.3049e-03 &3.5904e-03  &2.1280e-03 &1.7062e-03&1.2607e-03 \\
		Order& $\backslash$ & 1.0563 &0.9747& 0.9900 &1.0530&1.0307\\
		\bottomrule %[2pt]     
	\end{tabular}
\end{table}

\begin{table}[H]
	\centering
	\caption{Convergence results for the errors of $y$, $z$, $u$ for Example 5.2 on nonconvex meshes}
	\begin{tabular}{ccccccc}
		\toprule %[2pt]   
		$h$ &0.1457 & 0.0728 &0.0485 &0.0364 &0.0291& 0.0182 \\
		\midrule %[2pt]  
		$\Vert y-y_{h}\Vert_{0,\Omega}$ & 9.1257e-03 &3.6907e-03& 2.3042e-03&1.6721e-03 &1.3097e-03&7.8265e-04\\
		Order& $\backslash$ & 1.3060 &1.1618 &1.1148 &1.0946&1.0955 \\
		$\Vert z-z_{h}\Vert_{0,\Omega}$ & 1.5531e-02 &7.5826e-03 &4.9975e-03  &3.7015e-03 &2.9262e-03&1.7764e-03 \\
		Order& $\backslash$ & 1.0344 &1.0283 &1.0435 &1.0532&1.0620 \\
		$\Vert u-u_{h}\Vert_{0,\Omega}$ & 1.3144e-02 & 6.2386e-03 &4.1536e-03  &3.1075e-03 &2.4689e-03&1.5663e-03 \\
		Order& $\backslash$ & 1.0751 &1.0033 &1.0086 &1.0308&0.9682\\
		\bottomrule %[2pt]    
	\end{tabular}
\end{table}
\noindent\textbf{Example 5.3. }The third example is defined on the unit square $\Omega = [0,1]\times[0,1]$. The data is chosen as
\begin{equation*}
	\begin{aligned}
		y&=x_{1}^{2}x_{2}+\text{sin}(\pi x_{1})\text{sin}(\pi x_{2}),\\
		z&=-(x_{1}^{2}-x_{2}^{2}).
	\end{aligned}
\end{equation*}
We set 
\begin{equation*}
	A=
	\begin{bmatrix}
		x_{2}^{2}+1 & x_{1}x_{2} \\
		x_{1}x_{2} & x_{1}^{2}+1
	\end{bmatrix},
\end{equation*}
 $\gamma=1$, and the constrained control set is chosen as
\begin{equation*}
	U_{ad}:=\{u\in L^{2}(\Omega),\ 0\le u\le 0.5\ \mbox{a.e. in }\Omega \}.
\end{equation*}

In this example, the state equation has inhomogeneous Dirichlet boundary conditions and the coefficient matrix $A$ is a variable. We further need $A=(a_{ij}(\cdot))_{2\times2}\in (W^{1,\infty})^{2\times2}$ is a symmetric and positive definite matrix. Then, the process of proving a priori error is similar when $A$ is not a constant matrix. In Table 7-9 we present the related convergence results. We also observe first order convergence rates for both the state, adjoint state and control and that are optimal.
\begin{figure}[H]
	\begin{center}
		\includegraphics[width=0.3\linewidth]{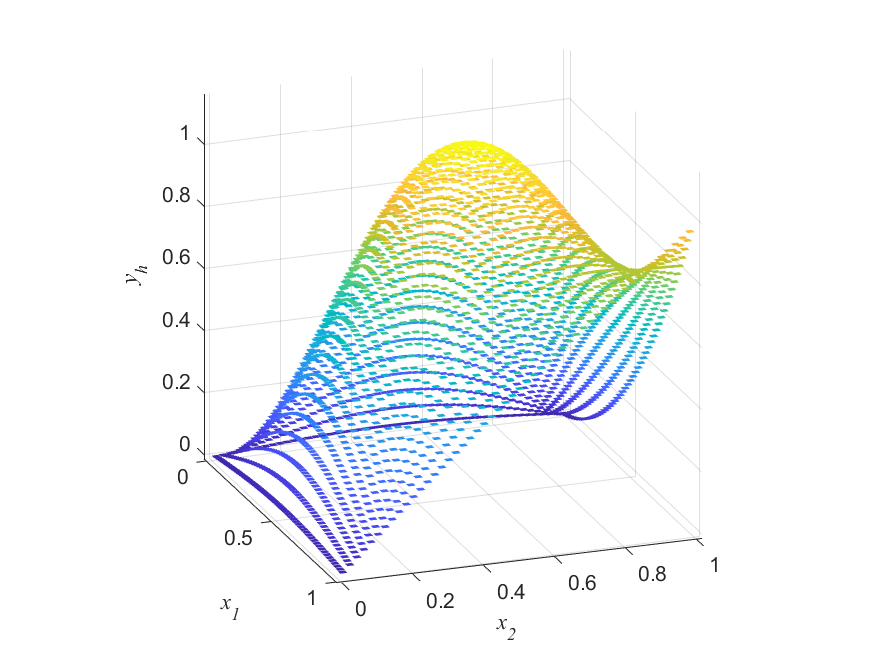}
		\includegraphics[width=0.3\linewidth]{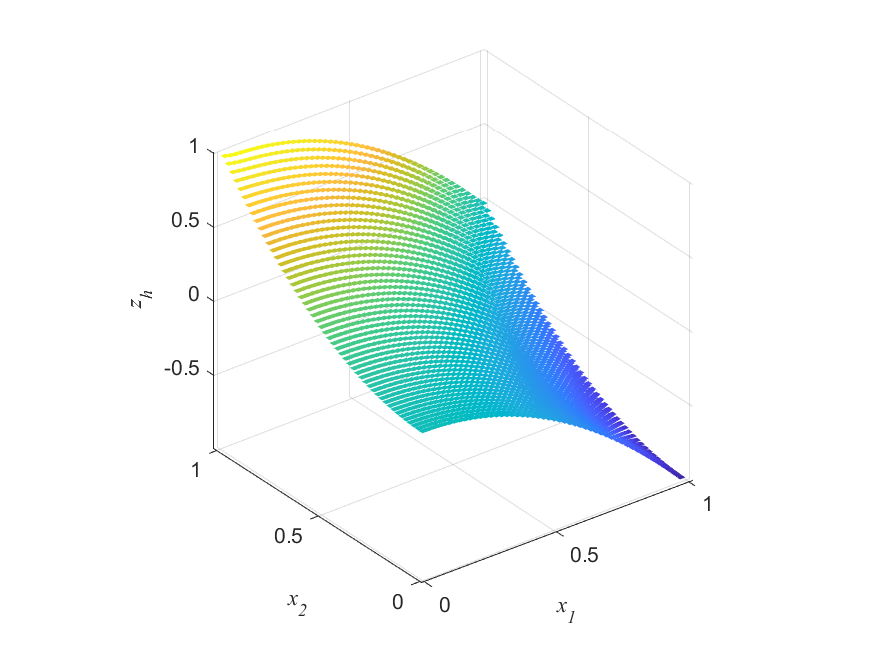}
		\includegraphics[width=0.3\linewidth]{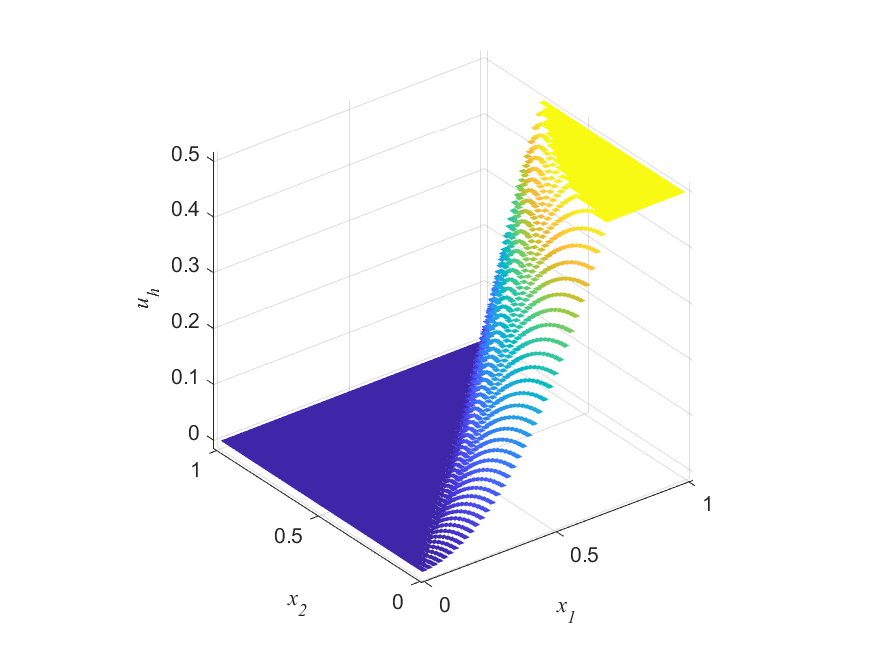}
		\caption{Different meshes}\label{fig41}		
	\end{center}
	\vspace{-1em}
\end{figure}
\begin{figure}[H]
	\begin{center}
		\includegraphics[width=0.3\linewidth]{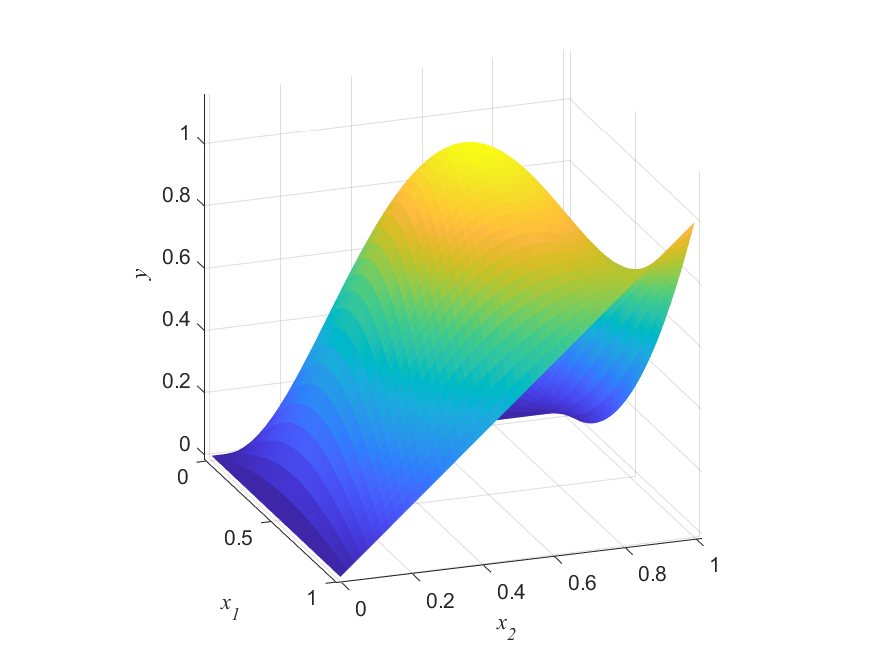}
		\includegraphics[width=0.3\linewidth]{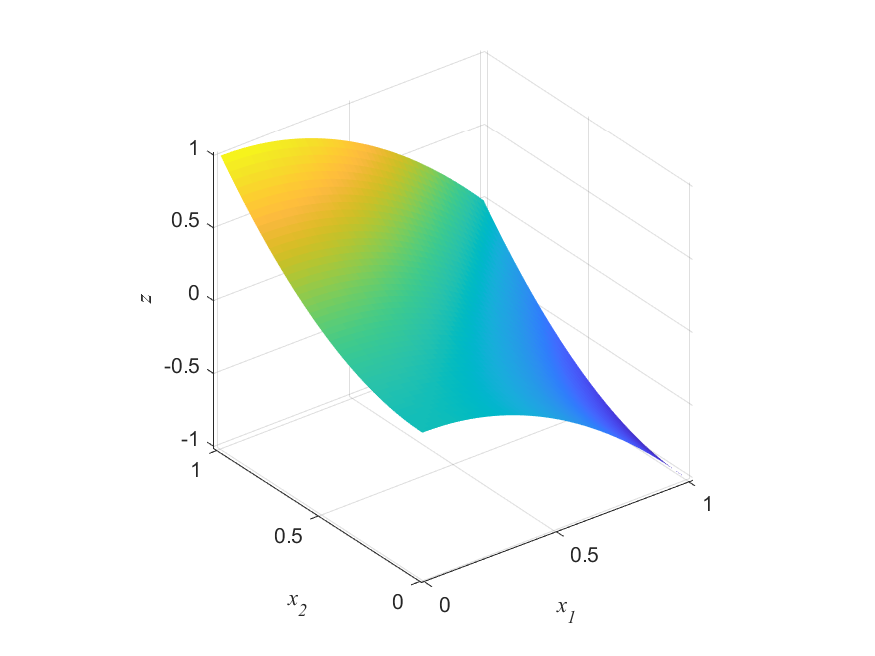}
		\includegraphics[width=0.3\linewidth]{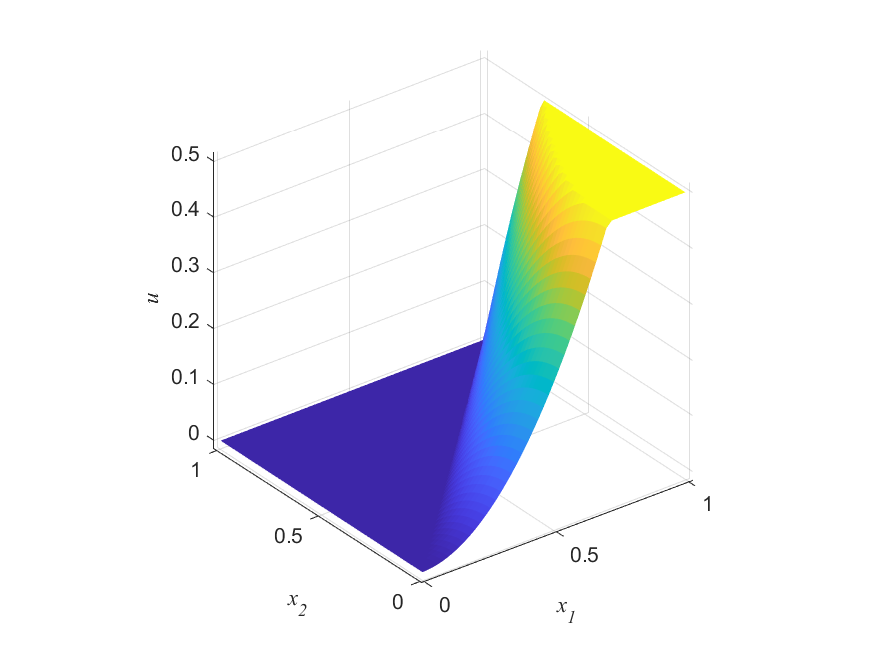}
		\caption{Different meshes}\label{fig42}		
	\end{center}
	\vspace{-1em}
\end{figure}
\begin{figure}[H]
	\begin{center}
		\includegraphics[width=0.3\linewidth]{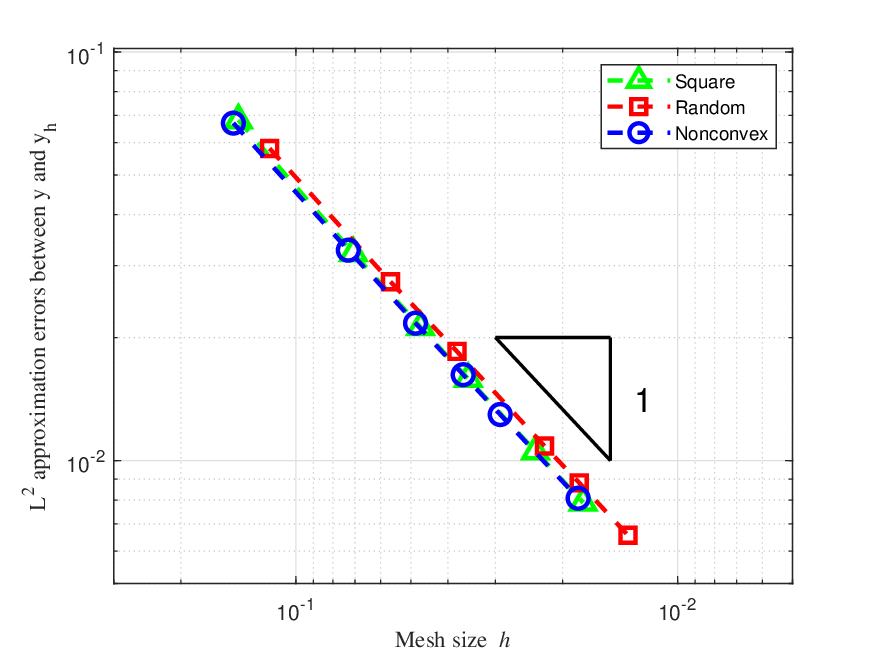}
		\includegraphics[width=0.3\linewidth]{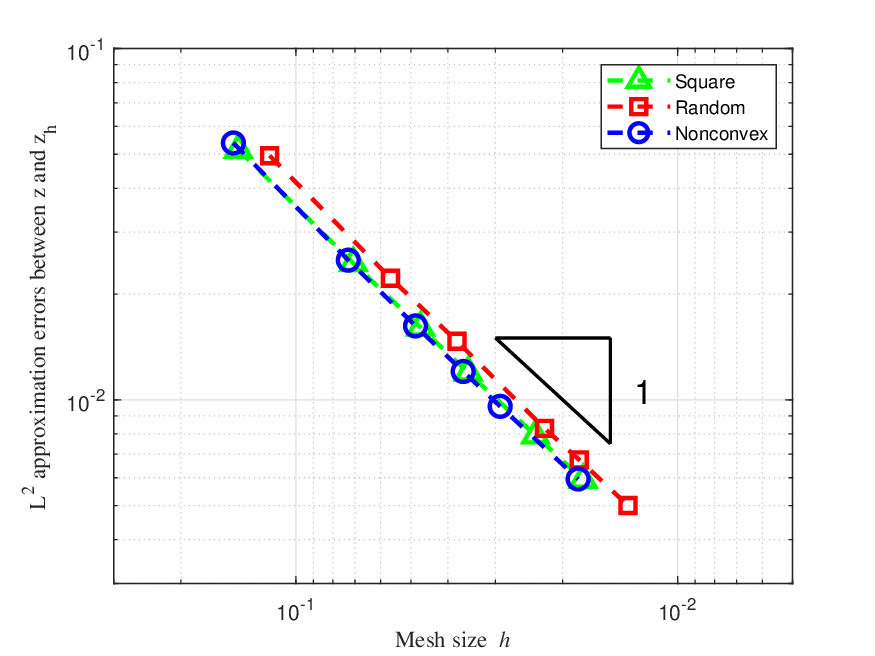}
		\includegraphics[width=0.3\linewidth]{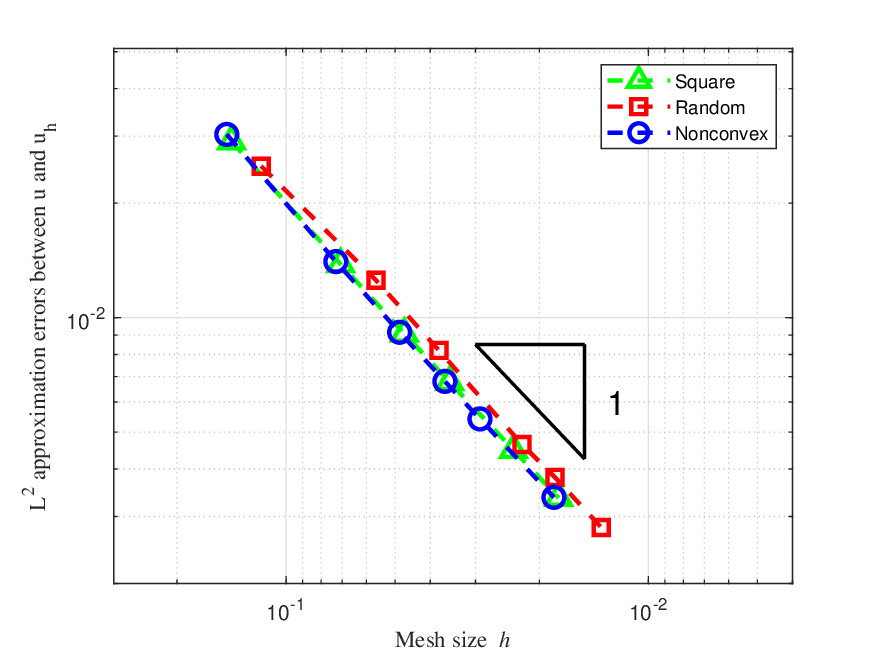}
		\caption{Different meshes}\label{fig4}		
	\end{center}
	\vspace{-1em}
\end{figure}

\section*{Acknowledgments}
The research was supported by the National Natural Science Foundation of China under Grant No. 11971276.

\end{document}